\documentclass[11pt]{amsart}
\usepackage{bbm}
\usepackage{amsmath, amssymb, amsthm, xypic, stmaryrd,graphicx}
\usepackage[all]{xy}
\usepackage{tikz-cd}
\usepackage{boxedminipage}
\usepackage{mathtools}
\numberwithin{equation}{section}

\DeclareMathOperator{\supp}{supp}

\DeclareMathOperator{\Pen}{Pen}

\DeclareMathOperator{\pr}{pr}

\DeclareMathOperator{\vol}{vol}
\DeclareMathOperator{\alg}{alg}
\DeclareMathOperator{\red}{red}
\DeclareMathOperator{\loc}{loc}

\DeclareMathOperator{\Av}{Av}

\DeclareMathOperator{\TAv}{\widetilde{Av}}

\DeclareMathOperator{\fp}{fp}

\DeclareMathOperator{\ind}{index}

\DeclareMathOperator{\Hom}{Hom}
\DeclareMathOperator{\End}{End}

\DeclareMathOperator{\Spin}{Spin}

\DeclareMathOperator{\SU}{SU}

\DeclareMathOperator{\im}{im}
\DeclareMathOperator{\dom}{dom}

\DeclareMathOperator{\Ad}{Ad}

\DeclareMathOperator{\TR}{TR}

\DeclareMathOperator{\tc}{tc}

\DeclareMathOperator{\restr}{restr}

\newcommand{\beq}[1]{\begin{equation} \label{#1}}
\newcommand{\eeq}{\end{equation}}

\newcommand{\bea}{\begin{eqnarray}}
\newcommand{\eea}{\end{eqnarray}}

\begin{document}

\theoremstyle{plain}
\newtheorem{theorem}{Theorem}[section]
\newtheorem{thm}{Theorem}[section]
\newtheorem{lemma}[theorem]{Lemma}
\newtheorem{proposition}[theorem]{Proposition}
\newtheorem{prop}[theorem]{Proposition}
\newtheorem{corollary}[theorem]{Corollary}
\newtheorem{conjecture}[theorem]{Conjecture}

\theoremstyle{definition}
\newtheorem{definition}[theorem]{Definition}
\newtheorem{defn}[theorem]{Definition}
\newtheorem{example}[theorem]{Example}
\newtheorem{remark}[theorem]{Remark}
\newtheorem{rem}[theorem]{Remark}

\newcommand{\C}{\mathbb{C}}
\newcommand{\R}{\mathbb{R}}
\newcommand{\Z}{\mathbb{Z}}
\newcommand{\N}{\mathbb{N}}

\newcommand{\norm}[1]{\left\lVert#1\right\rVert}

\newcommand{\Supp}{{\rm Supp}}

\newcommand{\field}[1]{\mathbb{#1}}
\newcommand{\bZ}{\field{Z}}
\newcommand{\bR}{\field{R}}
\newcommand{\bC}{\field{C}}
\newcommand{\bN}{\field{N}}
\newcommand{\bT}{\field{T}}

\newcommand{\cB}{{\mathcal{B} }}
\newcommand{\cK}{{\mathcal{K} }}
\newcommand{\cM}{{\mathcal{M} }}
\newcommand{\cF}{{\mathcal{F} }}
\newcommand{\cO}{{\mathcal{O} }}
\newcommand{\cE}{\mathcal{E}}
\newcommand{\cS}{\mathcal{S}}
\newcommand{\calL}{\mathcal{L}}

\newcommand{\HH}{{\mathcal{H}}}
\newcommand{\tilH}{\widetilde{\HH}}
\newcommand{\HX}{\HH_X}
\newcommand{\Hpi}{\HH_{\pi}}
\newcommand{\HHpi}{\HH \otimes \HH_{\pi}}
\newcommand{\Ltwopi}{L^2_{\pi}(X, \HHpi)}

\newcommand{\KK}{K\!K}

\newcommand{\D}{D \hspace{-0.27cm }\slash}
\newcommand{\Dsmall}{D \hspace{-0.19cm }\slash}

\newcommand{\mybigwedge}{\textstyle{\bigwedge}}

\newcommand{\CGmax}{C^*_{G, \max}}
\newcommand{\DGmax}{D^*_{G, \max}}
\newcommand{\CGred}{C^*_{G, \red}}
\newcommand{\CGalg}{C^*_{G, \alg}}
\newcommand{\DGalg}{D^*_{G, \alg}}
\newcommand{\CGker}{C^*_{G, \ker}}
\newcommand{\Cmax}{C^*_{\max}}
\newcommand{\Dmax}{D^*_{\max}}
\newcommand{\Cred}{C^*_{\red}}
\newcommand{\Calg}{C^*_{\alg}}
\newcommand{\Dalg}{D^*_{\alg}}
\newcommand{\Cker}{C^*_{\ker}}
\newcommand{\tilCalg}{\widetilde{C}^*_{\alg}}
\newcommand{\Cpiker}{C^*_{\pi, \ker}}
\newcommand{\Cpialg}{C^*_{\pi, \alg}}
\newcommand{\Cpimax}{C^*_{\pi, \max}}

\newcommand{\one}{\mathbbm{1}}

\newcommand{\Avpi}{\Av^{\pi}}
\newcommand{\tilAv}{\widetilde{\Av}}

\newcommand{\tilTR}{\widetilde{\TR}}

\newcommand{\Gtc}{\Gamma^{\infty}_{tc}}
\newcommand{\tilD}{\widetilde{D}}
\newcommand{\XoneH}{X^{\HH}_1}

\newcommand{\XX}{\mathfrak{X}}

\def\kt{\mathfrak{t}}
\def\kk{\mathfrak{k}}
\def\kp{\mathfrak{p}}
\def\kg{\mathfrak{g}}
\def\kh{\mathfrak{h}}

\newcommand{\pilamrho}{[\pi_{\lambda+\rho}]}

\newcommand{\Trestr}{\mathcal{T}_{\restr}}
\newcommand{\TdN}{\mathcal{T}_{d_N}}

\newcommand{\omG}{\om/\hspace{-1mm}/G}
\newcommand{\om}{\omega} \newcommand{\Om}{\Omega}

\newcommand{\QcwR}{quantization commutes with reduction}

\newcommand{\Spinc}{\Spin^c}

\def\kt{\mathfrak{t}}
\def\kk{\mathfrak{k}}
\def\kp{\mathfrak{p}}
\def\kg{\mathfrak{g}}
\def\kh{\mathfrak{h}}

\newcommand{\ddt}{\left. \frac{d}{dt}\right|_{t=0}}

\newenvironment{proofof}[1]
{\noindent \emph{Proof of #1.}}{\hfill $\square$}

%\newcommand{\Todo}{\textbf{To do}}

%\title{Equivariant Callias-type operators and the integration trace}
\title{Coarse geometry and Callias quantisation}
\author{Hao Guo}
\address{Texas A\&M University}
\email{haoguo@math.tamu.edu}
\author{Peter Hochs}
\address{University of Adelaide}
\email{peter.hochs@adelaide.edu.au}
\author{Varghese Mathai}
\address{University of Adelaide}
\email{mathai.varghese@adelaide.edu.au}

\maketitle

\begin{abstract}
Consider a proper, isometric action by a unimodular, locally compact group $G$ on a complete Riemannian manifold $M$.
For equivariant elliptic operators that are invertible outside a cocompact subset of $M$, we show that a localised index in the $K$-theory of the maximal group $C^*$-algebra of $G$ is well-defined. The approach is based on the use of maximal versions of equivariant localised Roe algebras, and many of the technical arguments in this paper are used to handle the ways in which they differ  from their reduced versions.

By using the maximal group $C^*$-algebra instead of its reduced counterpart, we can apply the trace given by integration over $G$ to recover an index defined earlier by the last two authors, and developed further by Braverman, in terms of sections invariant under the group action. 
This leads to refinements of index-theoretic obstructions to Riemannian metrics of positive scalar curvature on noncompact manifolds, and also on orbifolds and other singular quotients of proper group actions.
%
%
%
%The reduced version of the localised index allows one to refine numerical index-theoretic obstructions to positive scalar curvature on a smooth noncompact $\Spin$ manifold to $K$-theoretic obstructions. 
%
%The maximal localised index allows us to 
%
% extend this to orbifolds arising as quotients of proper actions by discrete groups, and more generally  to metrics of positive scalar curvature invariant under a proper group action.
%
% or an orbifold $X$ defined via Callias index theory, to obstructions in the $K$-theory of the maximal $C^*$-algebra
% of the fundamental group $\pi_1(X)$. 
%
 As a motivating application in another direction, we prove a version of Guillemin and Sternberg's quantisation commutes with reduction principle for equivariant indices of $\Spinc$ Callias-type operators.
\end{abstract}

\tableofcontents

\section{Introduction}

\subsection*{Background}

Let $M$ be a complete Riemannian manifold, and let $D$ be an elliptic differential operator on a vector bundle $E \to M$. The \emph{coarse index} \cite{Roe96} of $D$ lies in $K_*(C^*(M))$, the $K$-theory group of the \emph{Roe algebra} $C^*(M)$ of $M$. This Roe algebra is the closure in the operator norm of the algebra of locally compact, bounded operators on $L^2(E)$ that enlarge supports of sections by a finite amount. If $M$ is compact, then $C^*(M)$ is the algebra of compact operators, and the coarse index of $D$ is its Fredholm index. A strength of the coarse index is that it applies very generally, without any assumptions on compactness of $M$, or on the behaviour of $D$ at infinity. Coarse index theory has a range of applications, for example to Riemannian metrics of positive scalar curvature \cite{Schick14}, and to the Novikov conjecture \cite{Yu98, Yu00}. A central role here is played by the coarse Baum--Connes conjecture \cite{Roe93}.

The general applicability of the coarse index can come at the cost of computability. For that reason, it is worth looking for special cases, or variations, where a version of the coarse index is more explicit or computable.
One useful approach is Roe's \emph{localised coarse index} \cite{Roe16}. If $D^2$ is positive outside a subset $Z \subset M$ in a suitable sense, then Roe constructed a \emph{localised coarse index} 
\[
\ind^Z(D) \in K_*(C^*(Z)).
\]
The special case where $Z$ is compact is already of interest: then $D$ is Fredholm (by Theorem 2.1 in \cite{Anghel93a}), and its localised coarse index generalises the Gromov--Lawson index \cite{Gromov83},  the Atiyah--Patodi--Singer index on compact manifolds with boundary \cite{APS1}, and the index of Callias-type Dirac operators \cite{Anghel93, Bunke95, Kucerovsky01} $D = \tilde D + \Phi$, where $\tilde D$ is a Dirac operator, and $\Phi$ is a vector bundle endomorphism making $D$ invertible at infinity.

The localised coarse index was generalised to an equivariant version in \cite{GHM1}, for a proper, isometric action by a unimodular locally compact group $G$ on $M$, preserving all structure including $D$. Then, if $Z/G$ is compact, one obtains a \emph{localised equivariant index}
\beq{eq loc index intro}
\ind_{G, \red}^{\loc}(D) \in K_*(C^*_{\red}(G)),
\eeq
where $C^*_{\red}(G)$ is the reduced group $C^*$-algebra of $G$. The fact that this index lies in $K_*(C^*_{\red}(G))$ is useful, because that $K$-theory group is independent of $M$, and it is a very well-studied object that is central to many problems in geometry, topology and group theory. In particular, it is large enough to contain relevant group-theoretic information. And importantly, there is a range of traces and higher cyclic cocycles on subalgebras of $C^*_{\red}(G)$ that allows one to obtain a number from the index \eqref{eq loc index intro}, for which one can then find a topological expression. Examples of such expressions are the equivariant Atiyah--Patodi--Singer index theorems in \cite{CWXY, HWW, XieYu}, in the case of manifolds with boundary.

\subsection*{Results}

This paper is about the construction and application of a \emph{maximal} localised equivariant coarse index,  taking values in the $K$-theory of the {maximal} group $C^*$-algebra $C^*_{\max}(G)$
\beq{eq max loc index intro}
\ind_G^{\loc}(D) \in K_*(C^*_{\max}(G)).
\eeq
The first result in this paper is that this index is well-defined: see Theorem \ref{thm D reg} and Proposition \ref{prop Roe max}.

The  index \eqref{eq max loc index intro} has several advantages over \eqref{eq loc index intro}. From a general point of view, the natural map from $C^*_{\max}(G)$ to $C^*_{\red}(G)$ maps the index in $K_*(C^*_{\max}(G))$ to the one in $K_*(C^*_{\red}(G))$, so the former is a more refined invariant. 
On a more practical level, the integration map $I\colon L^1(G) \to \C$ extends to a trace on $C^*_{\max}(G)$ (not on $C^*_{\red}(G)$); this can also be viewed as an algebra homomorphism $I\colon C^*_{\max}(G) \to \C$. That means the induced map  $I_*\colon K_0( C^*_{\max}(G)) \to K_0(\C)$ on $K$-theory can be applied to the index \eqref{eq max loc index intro}, to yield the integer
\beq{eq I index intro}
I_*(\ind_G^{\loc}(D)) \in K_0(\C) = \Z.
\eeq
Morally, applying the integration trace $I$ should correspond to taking the $G$-invariant part of the equivariant index. The second result in this paper, Theorem \ref{thm invar index}, is that this is indeed the case in a precise sense: 
\beq{eq invar index intro}
I_*(\ind_G^{\loc}(D))  = \ind(D)^G, 
\eeq
where the right hand side is the Fredholm index of $D$ restricted to $G$-invariant sections that are square integrable transversally to orbits in a certain sense. The latter index was defined in \cite{Mathai13}, and developed further by Braverman \cite{Braverman14}.

In the example where $D_X$ is an elliptic operator on a possibly noncompact manifold $X$, invertible outside a compact set, and $D$ is its lift to the universal cover of $X$, \eqref{eq invar index intro} implies that $I_*$ maps the $\pi_1(X)$-equivariant, localised, maximal index of $D$ to the Fredholm index of $D_X$. This means that the index \eqref{eq max loc index intro} refines the Gromov--Lawson index, the index of Callias-type operators, as well as  the Atiyah--Patodi--Singer index. One application of this fact is that it leads to refinements of obstructions to Riemannian metrics of positive scalar curvature defined through the Gromov--Lawson and Callias indices on $\Spin$ manifolds. This is analogous to the way in which the image of $D$ under the analytic assembly map \cite{Connes94} for the maximal group $C^*$-algebra of $\pi_1(X)$ refines the index of $D_X$ in the case where $X$ is compact. The robustness of \eqref{eq invar index intro} allows one to generalise this to orbifolds, and more generally to metrics invariant under a proper group action.
This is in contrast to a version for the reduced group $C^*$-algebra in the compact case, where $I$ is replaced by the von Neumann trace, and the analogue of \eqref{eq invar index intro} only holds because the actions is free and the group is discrete. Furthermore, an analogue of Atiyah's $L^2$-index theorem used in the  reduced version is not available yet  for noncompact manifolds (but see Theorem 2.20 in \cite{Braverman18} for an analogue). Another approach to $K$-theoretic obstructions to positive scalar curvature, in terms of Callias operators, was developed in \cite{Cecchini16}. Explicit applications to positive scalar curvature will be explored in future work.

A completely different application that motivates the development of the index \eqref{eq max loc index intro} and Theorem \ref{thm invar index}.
 is a version of Guillemin and Sternberg's \emph{quantisation commutes with reduction} principle \cite{Guillemin82} for Callias-type $\Spinc$-Dirac operators. That principle was initially stated and proved for compact K\"ahler and symplectic manifolds \cite{Meinrenken98, Meinrenken99, Paradan01, Zhang98}. This principle was extended in various directions, including results for proper actions by possibly noncompact groups, with possibly noncompact orbit spaces, see \cite{Mathai13} for the symplectic case and  \cite{HM14} for $\Spinc$-manifolds. The index, or quantisation, used in those papers, was defined just in terms of sections invariant under the group action. Furthermore, the index was only well-defined after a suitable order zero term was added to the operator in question. The first of these issues was partially remedied in \cite{HSII}, where the quantisation commutes with reduction principle was proved for an index with values in the completed representation ring of a maximal compact subgroup of $G$.
 
Since the work of Paradan and Vergne \cite{Paradan17}, the quantisation commutes with reduction principle is known to be a general property of equivariant indices of $\Spinc$-Dirac operators in general, and not just of geometric quantisation in the narrow sense. For a Callias-type operator $D = \tilde D + \Phi$, where $\tilde D$ is a $\Spinc$-Dirac operator, the third result in this paper, Theorem \ref{thm QR=0}, states that the quantisation commutes with reduction principle holds, in the sense that 
\beq{eq QR=0 intro}
I_*(\ind_G^{\loc}(\tilde D + \Phi)) = \ind(D_0),
\eeq
where $D_0$ is a Dirac operator on a \emph{reduced space} $M_0$, a $\Spinc$-analogue of a reduced space in symplectic geometry, for high enough powers of the determinant line bundle of the $\Spinc$-structure. In this setting, the use of the maximal localised coarse index allows us to prove such a result in the setting of noncompact groups and orbit spaces, for a truly equivariant index in $K_0(C^*_{\max}(G))$, which is defined without the need of an added term. 

The equality \eqref{eq QR=0 intro} already appears to be new in the case where $G$ is compact. Then $\tilde D + \Phi$ is Fredholm, and has an equivariant index in the usual sense. In this case, a version of the shifting trick in symplectic geometry applies to yield information about the multiplicities in that index of all irreducible representations of $G$. Such a result would apply for example to an equivariant version of Callias' treatment \cite{Callias78} of fermions in the field of magnetic $\SU(2)$-monopoles.

% important trace: I

\subsection*{Techniques used}

The key ingredient in the construction of the index \eqref{eq max loc index intro} is the notion of a \emph{maximal localised equivariant Roe algebra} for arbitrary unimodular, locally compact groups. This involves the notion of an \emph{admissible module}, which was defined in \cite{Yu10} for discrete groups, and in \cite{GHM1} in general. In the non-equivariant, non-localised case, the natural maximal norm for such algebras was shown to be well-defined in \cite{GWY}. In the equivariant, localised case, this is less clear, and getting around this is a step in the construction of the algebras we need.

The construction of the index \eqref{eq max loc index intro} is very different form the construction of the reduced version \eqref{eq loc index intro} in \cite{GHM1}. Instead of viewing $D$ as an unbounded operator on $L^2(E)$, we view it as an unbounded operator on a maximal localised equivariant Roe algebra $A$, viewed as a Hilbert $C^*$-module over itself. The reason for this is that the localisation results in \cite{Roe16} that make the definition of the localised coarse index possible do not directly carry over to the norm on the maximal Roe algebra. Indeed, it is not even clear if the operators involved lie in the unlocalised maximal Roe algebra, let alone if they localise in a suitable way.

We prove versions of Roe's localisation results for $D$ as an operator on $A$, thus allowing us to define \eqref{eq max loc index intro}. To do this we prove that the functional calculus for such operators on $A$ is well-defined. This was done in \cite{GXY} for the uniform maximal Roe algebra; in our setting it works for usual maximal Roe algebras due to localisation at a cocompact set.

To prove the equality \eqref{eq invar index intro}, we use various \emph{averaging maps}, which map $G$-equivariant operators on $M$ to operators on $M/G$. Comparing such maps for operators on $L^2(E)$ and on $A$ to the integration trace $I$ then leads to a proof of 
\eqref{eq invar index intro}.

Using \eqref{eq invar index intro}, we see that the left hand side of \eqref{eq QR=0 intro} equals a more concrete index in terms of $G$-invariant sections. For the latter index, we obtain localisation estimates that allow us to show that this index equals the right hand side of \eqref{eq QR=0 intro}. These localisation estimates  build on those in \cite{Mathai13, HM14, Mathai10, Zhang98}, but a fundamental difference is that we now need the key deformation term to go to zero at infinity, rather than grow towards infinity.

\subsection*{Outline of this paper}

We start by defining equivariant localised maximal Roe algebras in Section \ref{sec prelim}. That allows us to state the three results in the paper mentioned above, in Section \ref{sec results}. Well-definedness of the index \eqref{eq max loc index intro} is proved in Section \ref{sec reg loc}. To prepare for the proof of \eqref{eq invar index intro}, we construct several averaging maps in section \ref{sec avg}. In Section \ref{sec pf invar index}, we use these maps to prove \eqref{eq invar index intro}. We conclude this paper by using \eqref{eq invar index intro} and some localisation estimates to prove \eqref{eq QR=0 intro} in Section \ref{sec pf QR=0 1}.

%\Todo: things to mention
%\begin{itemize}
%\item
% Goal is to construct maximal version of index in first paper
% Motivation is that we can then apply the integration trace to recover indices on the quotient.
% Make this precise by showing we recover index from first paper with Mathai
% Application/illustration: [Q,R]=0. Call this the main result/main application? This uses everything else in the paper.
%\item
% Cite Cecchini. 
%\item Get refined psc obstructions. Mention Rosenberg invariant, which is a further refinement of assembly map applied to lift to univ. cover. Mention Hao's paper with Zhizhang and Guoliang.
% Mention APS application of reduced version.
%\item
% $[Q,R]=0$: quant of non-cocompact actions defined with deformation term that makes [Q,R]=0 work. No we don't need this term to define the index, just to prove the equality (As in Tian-Zhang and Mathai-Zhang)
% Compact group case already new. Very explicit, and have shifting trick in that case.
%\item
%Good thing about index in $K$-theory of group $C^*$-algebra is that (1) it's independent of the space etc. and (2) there are many cocycles/traces to pair with to get index theorems and hence topological obstructions to psc. Although, for the coarse index we also have the partitioned manifold index theorem. This is future work.
%\item (Also in main text?) Reason for considering operators on Hilbert modules is that that allows us to generalise Roes localisation lemmas. Not obvious for func calc of bounded operators on $L^2(E)$.
%\end{itemize}

\subsection*{Acknowledgements}

The authors are grateful to Rufus Willett, Zhizhang Xie and Guoliang Yu for their helpful advice. Varghese Mathai was supported by funding from the Australian Research Council, through the Australian Laureate Fellowship FL170100020. Hao Guo was supported in part by funding from the National Science Foundation under grant no.\ 1564398.

\section{Equivariant localised maximal Roe algebras} \label{sec prelim}

Throughout this paper, $G$ will be a unimodular, locally compact group, with a Haar measure  $dg$. 

We assume that $G$ admits a left-invariant distance function $d_G$ for which there are $a,b>0$ such that for all $r>0$, any ball in $G$ of radius $r$ has volume at most $ae^{br}$. This is the case, for example, if the connected component $G_0<G$ is a Lie group, and $G/G_0$ is finitely generated. This volume growth condition is used in the proof of Lemma \ref{lem L1G2infty}, which is a step in the proof of Theorem \ref{thm D reg}, which in turn is the basis of the functional calculus of operators on Hilbert $C^*$-modules that we use.

\subsection{Equivariant $C_0(X)$-modules} \label{sec C0X mod}

Let $(X, d)$ be a metric space in which all closed balls are compact. Suppose that $G$ acts properly and isometrically on $X$.

A \emph{$G$-equivariant $C_0(X)$-module} is a Hilbert space $\HH_X$ equipped with a unitary representation $\pi$ of $G$, and a $*$-homomorphism $\rho\colon C_0(X) \to \cB(\HH_X)$, such that for all $g \in G$ and $\varphi \in C_0(X)$,
\[
\pi(g) \rho(\varphi) \pi(g)^{-1} = \rho(g\cdot \varphi).
\]
Here $(g\cdot \varphi)(x) = \varphi(g^{-1}x)$, for all $x \in X$. We will omit the representations $\pi$ and $\rho$ from the notation, and for example write $\varphi \cdot \xi := \rho(\varphi)\xi$, for $\varphi \in C_0(X)$ and $\xi \in \HH_X$.

Fix a $G$-equivariant $C_0(X)$-module $\HH_X$.
Let $\cB(\HH_X)^G$ be the algebra of $G$-equivariant bounded operators on $\HH_X$. An operator $T \in \cB(\HH_X)$ is said to be \emph{locally compact} if for all $\varphi \in C_0(X)$, the operators $\varphi T$ and $T \varphi$ are compact. And $T$ has \emph{finite propagation} if there is an $r>0$ such that for all $\varphi, \psi \in C_0(X)$ whose supports are at least a distance $r$ apart,
\[
\varphi T \psi = 0.
\]
In that case, the infimum of such numbers $r$ is the \emph{propagation} of $T$. The \emph{$G$-equivariant reduced Roe algebra} of $X$ with respect to $\HH_X$ is the closure in the operator norm of the algebra of locally compact operators in $\cB(\HH_X)^G$ with finite propagation. In this paper, we will use an algebra that differs from the equivariant reduced Roe algebra in two ways: we consider a \emph{localised} version, and complete it in a \emph{maximal} norm.

A relevant example of a $G$-equivariant $C_0(X)$-module is the space $L^2(E)$ of square integrable sections of a $G$-equivariant, Hermitian vector bundle $E \to X$, with respect to  a $G$-invariant measure $dx$ on $X$. The algebra $C_0(X)$ acts on $L^2(E)$ by pointwise multiplication, and $G$ acts in the usual way. Consider the vector bundle $\Hom(E) := E \boxtimes E^* \to X \times X$. Let $C^*_{\ker}(X; L^2(E))^G$ be the algebra of locally compact operators $T \in \cB(L^2(E))^G$ with finite propagation, for which there is a bounded, measurable\footnote{One can also work with continuous sections; the main reason we use measurable sections is that the map $\kappa \mapsto \tilde\kappa$ in Lemma \ref{lem MM GG} does not preserve continuity.} section $\kappa$ of $\Hom(E)$ such that for all $s\in L^2(E)$ and $x \in X$,
\[
(Ts)(x) = \int_{X} \kappa(x,x')s(x')\, dx'.
\]
We  will identify such operators with their kernels $\kappa$.

\subsection{Admissible modules and the maximal Roe algebra}

In Definition 2.2 in \cite{Yu10}, the notion of an \emph{admissible} $\Gamma$-equivariant $C_0(X)$-module was introduced, for discrete groups $\Gamma$. In Definition 2.4 in \cite{GHM1}, this was extended to general unimodular, locally compact groups $G$, in the case where $X/G$ is compact. The main difference between the discrete and general group case is the role played by local slices in the sense of Palais \cite{Palais61} in the non-discrete case.

Suppose that $X/G$ is compact.
A $G$-equivariant $C_0(X)$-module  $\HH_X$ is defined to be admissible if there is a $G$-equivariant, unitary isomorphism
\[
\HH_X \cong L^2(G) \otimes \HH,
\]
for a Hilbert space $\HH$, such that locally compact operators on $\HH_X$ are mapped to locally compact operators on $L^2(G) \otimes \HH$, and operators with finite propagation are mapped to operators with finite propagation, in both cases with respect to the pointwise action by $C_0(G)$.

The point of using admissible modules is that the resulting equivariant Roe algebras encode the relevant group-theoretic information. It is clear that such information may be lost in the example where $X$ is a point, acted on trivially by a compact group, and one uses the non-admissible module $\C$.

By Theorem 2.7 in \cite{GHM1}, an example of an admissible module is
%\footnote{Here it is assumed that at least one of the sets $G/K$, for $K<G$ maximal compact, or $X/G$ is infinite. This is always the case in the setting we consider.}
 $L^2(E) \otimes L^2(G)$. Here $E \to X$ is as at the end of the previous subsection, $C_0(X)$ acts pointwise on the factor $L^2(E)$, and $G$ acts diagonally, with respect to the left regular representation of $G$ in $L^2(G)$. By definition of admissibility, we have an isomorphism 
 \beq{eq geom adm}
 L^2(E) \otimes L^2(G) \cong L^2(G) \otimes \HH
 \eeq
 with the properties above. Let $C^*_{\ker}(X; L^2(E) \otimes L^2(G))^G$ be the algebra of $G$-equivariant, locally compact operators on $L^2(E) \otimes L^2(G)$ with finite propagation, given by bounded, measurable kernels 
 \[
 \kappa_G\colon G \times G \to \cK(\HH)
 \]
via the isomorphism \eqref{eq geom adm}. Explicitly, for such a $\kappa_G$, the corresponding operator $T$ is defined by
\[
(T(f\otimes \xi))(g) = \int_G f(g')\kappa(g,g')\xi\, dg',
\]
for  $f \in L^2(G)$, $\xi \in \HH$
and $g \in G$. If $X/G$ is compact, then Theorem 2.11 in \cite{GHM1} states that $C^*_{\ker}(X; L^2(E) \otimes L^2(G))^G$ is isomorphic to a dense subalgebra of $C^*(G) \otimes \cK(\HH)$, where $C^*(G)$ is either the reduced or maximal group $C^*$-algebra of $G$. (To be precise, $C^*_{\ker}(X; L^2(E) \otimes L^2(G))^G$ is isomorphic to the convolution algebra of compactly supported, bounded, measurable functions on $G$ with values in the algebra of compact operators on $\HH$.) This implies that the \emph{maximal norm} of an element $\kappa \in C^*_{\ker}(X; L^2(E) \otimes L^2(G))^G$,
\beq{eq def max norm}
\|\kappa\|_{\max} := \sup_{\eta} \|\eta(\kappa)\|_{\cB(\HH_{\eta})},
\eeq
where the supremum is over all $*$-representations 
$$\eta\colon C^*_{\ker}(X; L^2(E) \otimes L^2(G))^G \to \cB(\HH_{\eta}),$$ 
is finite. Then
\beq{eq max norm gp Cstar}
\|\kappa\|_{\max} = \|\kappa_G\|_{C^*_{\max}(G) \otimes \cB(\HH)},
\eeq
%This norm is equal to the tensor product norm on $C^*_{\max} (G)\otimes \cK(\HH)$. So
so the completion of $C^*_{\ker}(X; L^2(E) \otimes L^2(G))^G$ in the maximal norm equals
\beq{eq max Roe gp alg}
C^*_{\max}(X; L^2(E) \otimes L^2(G))^G \cong C^*_{\max}(G) \otimes \cK(\HH).
\eeq
Since this algebra is independent of the admissible module used, we will denote it by
\[
C^*_{\max}(X)^G := C^*_{\max}(X; L^2(E) \otimes L^2(G))^G. 
\]

\begin{remark}
In the case where $G$ is trivial, and $X$ is not assumed to be compact but is only assumed to have bounded geometry, finiteness of the maximal norm \eqref{eq def max norm} was proved by Gong, Wang and Yu, see Lemma 3.4 in \cite{GWY}. See also Lemma 1.10 in \cite{Spakula13}. This generalises directly to free cocompact actions by discrete groups, see Lemma 3.16 in \cite{GWY}. For the case of non-cocompact actions, see \cite{GXY}.
\end{remark}

\subsection{The map $\oplus\,0$ and the maximal norm for non-admissible modules}
\label{sec plus 0}

We will use a completion of $C^*_{\ker}(X; L^2(E))^G$ in a version of the maximal norm. It is unclear a priori if an analogue of the supremum \eqref{eq def max norm} is finite, however. We therefore define the norm we use via an embedding of $C^*_{\ker}(X; L^2(E))^G$ into $C^*_{\ker}(X; L^2(E) \otimes L^2(G))^G$, which has a well-defined maximal norm if $X/G$ is compact, as we saw at the end of the previous subsection. 

Let $C^*_{\alg}(X; \HH_X)^G$ be the algebra of bounded, $G$-equivariant, locally compact operators on an equivariant  $C_0(X)$-module $\HH_X$, with finite propagation.
In Section 3.2 in \cite{GHM1}, a map
\beq{eq def plus 0}
\oplus\,0\colon C^*_{\alg}(X; L^2(E))^G \to  C^*_{\alg}(X; L^2(E) \otimes L^2(G))^G 
\eeq
is defined as follows. Let $\chi \in C(X)$ be a function whose support has compact intersections with all $G$-orbits, and has the property that for all $x \in X$,
\beq{eq chi cutoff}
\int_{G}\chi(gx)^2\, dx = 1.
\eeq
(The integrand is compactly supported by properness of the action.)  We can and will choose such a function that is bounded above by $1$.
Such a function will be called a \emph{cutoff function}.
The map $j\colon L^2(E) \to L^2(E) \otimes L^2(G)$, given by
\beq{eq def j}
(j(s))(x,g) = \chi(g^{-1}x)s(x)
\eeq
for $s \in L^2(E)$, $x \in X$ and $g \in G$,
is an isometric, $G$-equivariant embedding. Let $p\colon L^2(E) \otimes L^2(G) \to j(L^2(E))$ be the orthogonal projection. The map
\beq{eq def plus 0 B}
\oplus\,0\colon \cB(L^2(E)) \to \cB(L^2(E) \otimes L^2(G))
\eeq
that maps $T \in \cB(L^2(E))$ to $jTj^{-1}p$ is an injective $*$-homomorphism, and preserves equivariance, local compactness, and finite propagation. Hence it restricts to an injective $*$-homomorphism \eqref{eq def plus 0}. (The notation $\oplus\,0$ reflects the fact that $T\oplus 0$ equals $jT j^{-1}$ on the image of $j$, and zero on its orthogonal complement.)

\begin{lemma}
The map \eqref{eq def plus 0} maps $C^*_{\ker}(X; L^2(E))^G$ into $C^*_{\ker}(X; L^2(E) \otimes L^2(G))^G$.
\end{lemma}
\begin{proof}
The key point  is that \eqref{eq def plus 0} maps kernels with finite propagation in $M$ to kernels with finite propagation in $G$; that follows from the explicit expression for this map in Lemma \ref{lem MM GG}. (See also (11) in \cite{GHM1}.)
\end{proof}
If $X/G$ is compact, then for $\kappa \in C^*_{\ker}(X; L^2(E))^G$, we define its \emph{maximal norm} as 
\beq{eq max norm L2E}
\|\kappa\|_{\max} := \|\kappa \oplus 0\|_{\max}.
\eeq
For different choices of $j$ used in the definition of the map $\oplus 0$, the corresponding operators $\kappa \oplus 0$ are conjugate via isometries. This implies that the norm $\|\cdot \|_{\max}$ does not depend on the choice of $j$.
We denote the completion of $C^*_{\ker}(X; L^2(E))^G$ in this norm by $C^*_{\max}(X; L^2(E))^G$
\begin{remark}
On any $*$-algebra, one can define a maximal norm analogous to \eqref{eq def max norm}, if this supremum is finite for all elements of the algebra. The norm \eqref{eq max norm L2E} is \emph{not} this maximal norm. The reason why we use the norm  \eqref{eq max norm L2E} instead of the generally defined maximal norm is that for the algebra $C^*_{\ker}(X; L^2(E))^G$, it does not seem obvious a priori if the supremum in \eqref{eq def max norm} is finite. For free actions by discrete groups, this is shown in Lemmas 3.4 and 4.13 in \cite{GWY}. Furthermore, the equality \eqref{eq max Roe gp alg}, as well as the key ingredient for our use of functional calculus on Hilbert $C^*$-modules, Theorem \ref{thm D reg}, are true for
the norm \eqref{eq max norm L2E}.
\end{remark}
\begin{remark}
In the case of reduced Roe algebras, defined with respect to the operator norm for a $C_0(X)$-module, the algebra $C^*_{\ker}(X; L^2(E) \otimes L^2(G))^G$ is dense in $C^*_{\alg}(X; L^2(E) \otimes L^2(G))^G$. See Proposition 5.11 in \cite{GHM1}. In that case, kernels and operators can be used more or less interchangeably, but this is less clear for the maximal completions we use here. 
\end{remark}

% Todo: maximal norm on adm module, \oplus\,0 map, maxial norm for L^2(E), localised versions. Comment that for reduced Roe alg, kernels are dense.

%Consider the
%action by $G$ on sections $s$ of this bundle defined by
%\[
%(g\cdot \kappa)(x,x') = g \circ \kappa(g^{-1}x, g^{-1}x') g^{-1},
%\]
%for all $g \in G$ and $x,x' \in X$.

\subsection{Localised maximal Roe algebras} \label{sec loc max Roe}

Let $Z \subset X$ be a $G$-invariant subset. 
Let $\HH_X$ be a $G$-equivariant $C_0(X)$-module.
An operator $T \in \cB(\HH_X)^G$ is \emph{supported near $Z$} if there is an $r>0$ such that for all $\varphi$ whose support is at least a distance $r$ away from $Z$, the operators $\varphi T$ and $T \varphi$ are zero. Let $C^*_{\ker}(X; Z, \HH_X))^G$ be the algebra  of elements of $C^*_{\ker}(X; \HH_X)^G$ supported near $Z$. 

For $r\geq 0$ and any subset $Y \subset X$, we write
\[
\Pen(Y, r) := \{x \in X; d(x,Y) \leq r\}.
\]
Then we have a natural isomorphism
\beq{eq Cker inj lim}
C^*_{\ker}(X; Z, \HH_X)^G = \varinjlim_{r} C^*_{\ker}(\Pen(Z,r);  \HH_X)^G.
\eeq

Now suppose that $Z/G$ is compact. The algebra $C^*_{\ker}(X; Z, \HH_X)^G$  is then independent of $Z$, as long as $Z/G$ is compact. For this reason, we write
\[
C^*_{\ker}(X;  \HH_X)^G_{\loc} := C^*_{\ker}(X; Z, \HH_X)^G.
\]
For every $r>0$, we have the norm $\|\cdot\|_{\max}$ on $C^*_{\ker}(\Pen(Z,r);  L^2(E))^G$. 
Let $\|\cdot \|_{\max}$ be the resulting norm on $C^*_{\ker}(X, L^2(E))^G_{\loc}$ via \eqref{eq Cker inj lim}.
\begin{definition} \label{def max loc Roe}
The \emph{localised, $G$-maximal equivariant Roe algebra of $X$ for $L^2(E)$}, denoted by $C^*_{\max}(X; L^2(E))^G_{\loc}$, is the completion of $C^*_{\ker}(X, L^2(E))^G_{\loc}$ in the norm $\|\cdot \|_{\max}$.

The \emph{localised, $G$-maximal equivariant Roe algebra of $X$}, denoted by $C^*_{\max}(X)^G_{\loc}$, is the completion of $C^*_{\ker}(X; L^2(E) \otimes L^2(G))^G_{\loc}$ in the norm $\|\cdot \|_{\max}$.
\end{definition}
By construction, $C^*_{\max}(X; L^2(E))^G_{\loc}$ is isometrically embedded into $C^*_{\max}(X)^G_{\loc}$. By \eqref{eq max Roe gp alg} and \eqref{eq Cker inj lim}, 
\[
C^*_{\max}(X)^G_{\loc} \cong C^*_{\max}(G) \otimes \cK(\HH).
\]

%For brevity, we will write $A:= C^*_{\max}(X; Z, L^2(E))^G$ when it is clear what $X$, $Z$, $G$ and $E$ are.

\section{Results} \label{sec results}

Our first result is the fact that a maximal version of the localised equivariant index of \cite{GHM1} is well-defined, see Theorem  \ref{thm D reg} and Proposition \ref{prop Roe max} and Definition \ref{def loc index}.
We will show that that index is an equivariant refinement of the index defined in terms of invariant sections in \cite{Braverman14, Mathai13, Mathai10}, see Theorem \ref{thm invar index}. The quantisation commutes with reduction results for proper, non-cocompact actions in \cite{Mathai13, HM14} only involved  sections invariant under a group action. In Theorem \ref{thm QR=0}, we generalise this to the equivariant index of Definition \ref{def loc index}, in the case of Callias-type $\Spinc$-Dirac operators. 

\subsection{The localised maximal equivariant index} \label{sec loc index}

From now on, we suppose that $X = M$, a complete Riemannian manifold, and and that $d$ is the Riemannian distance corresponding to a $G$-invariant Riemannian metric. We suppose that $E\to M$ is a smooth, $G$-equivariant, Hermitian vector bundle and $D$ a symmetric, first order, elliptic, $G$-equivariant differential operator on sections of $E$. 
Suppose that $D$ has finite propagation speed, i.e.\ if $\sigma_D$  is its principal symbol, then
\[
\sup\{\|\sigma_D(\xi)\|; \xi \in T^*M, \|\xi\| = 1\|\} < \infty.
\]
Then $D$ is essentially self adjoint as an unbounded operator on $L^2(E)$, see Proposition 10.2.11 in \cite{Higson00}.

Let $Z\subset M$ be a closed, cocompact $G$-invariant subset. Let $C^{\infty}_{\ker}(M; L^2(E))^G_{\loc}$ be the algebra of smooth kernels in $C^*_{\ker}(M;  L^2(E))^G_{\loc}$. Then $D$ acts on $\kappa \in C^{\infty}_{\ker}(M; L^2(E))^G_{\loc}$ by
\[
(D\kappa)(m,m') := (D\otimes 1_{E^*_{m'}}(\kappa(\relbar, m')))(m).
\]
Here we used the fact that for every $m' \in M$, $\kappa(\relbar, m')$ is a smooth section of $E \otimes E^*_{m'}$.

For $A$ a $C^*$-algebra and $\mathcal{M}$ a Hilbert $A$-module, we write $\calL_A(\mathcal{M})$ and $\mathcal{K}_A(\mathcal{M})$ for the $C^*$-algebras of bounded adjointable operators and compact operators on $\mathcal{M}$, respectively. We can view $A$ as a right Hilbert $C^*$-module over itself, with $A$-valued inner product
\beq{eq A inner prod A}
\langle a,b\rangle := a^*b,
\eeq
for $a,b \in A$.
Then $\cK_A(A)\cong A$, with the isomorphism being given by identifying the operator
	$$\theta_{a,b}:c\mapsto a\langle b,c\rangle$$
with left multiplication by $ab^*$. %(\Todo: check this is left, and not right composition. Or do both work?) 
We also have that $\calL_A(A)$ is the multiplier algebra of $\cK_A(A)$.
%In this case the bounded adjointable and compact operators are denoted by $\calL_A(A)$ and $\cK_A(A)$ respectively. 

To simplify notation, we will from now on use $A$ to denote the $G$-maximal, localised equivariant Roe algebra $C^*_{\max}(M;  L^2(E))^G_{\loc}$. Then $A$ is a Hilbert module over itself. 
We will use functional calculus for self-adjoint, regular operators on the Hilbert $A$-module $A$. (For a uniform version of the maximal Roe algebra, this was developed in \cite{GXY}.) This functional calculus applies to $D$ because of the following result.
\begin{theorem} \label{thm D reg}
The unbounded operator $D$ on the Hilbert $A$-module $A$ is essentially self-adjoint and regular.
\end{theorem}
This theorem is proved in Subsection \ref{sec proof prop D reg}.
Because of Theorem \ref{thm D reg}, we can apply the following general result (see \cite{Kucerovsky}, \cite{Hanke-Pape-Schick} Theorem 3.1 and \cite{Ebert} Theorem 1.19) about functional calculus on Hilbert $C^*$-modules to the self-adjoint closure of $D$.
\begin{theorem}\label{thm:functionalcalculus}
Let $B$ be a $C^*$-algebra and $\mathcal{M}$ a Hilbert $B$-module. Let $C(\mathbb{R})$ be the $*$-algebra of complex-valued continuous functions on $\mathbb{R}$. For any regular, essentially self-adjoint operator $T$ on $\mathcal{M}$, there is a $*$-preserving linear map
	$$\pi_T:C(\mathbb{R})\rightarrow\mathcal{R}_B(\mathcal{M}),$$
with values in the set $\mathcal{R}_B(\mathcal{M})$ of regular operators on $\mathcal{M}$, such that:	\begin{enumerate}
		\item[(i)] $\pi_T$ restricts to a $*$-homomorphism $\pi_T:C_b(\mathbb{R})\rightarrow\mathcal{L}_B(\mathcal{M})$;
		\item[(ii)] If $|f(t)|\leq|g(t)|$ for all $t\in\mathbb{R}$, then $\textnormal{dom}(\pi_T(g))\subseteq\textnormal{dom}(\pi_T(f))$;
		\item[(iii)] If $(f_n)_{n\in\mathbb{N}}$ is a sequence in $C(\mathbb{R})$ for which there exists $F\in C(\mathbb{R})$ such that $|f_n(t)|\leq |F(t)|$ for all $t\in\mathbb{R}$, and if $f_n$ converge to a limit function $f \in C(\R)$ uniformly on compact subsets of $\mathbb{R}$, then $\pi_T(f_n)x\mapsto\pi_T(f)x$ for each $x\in\textnormal{dom}(\pi_T(f))$;
		\item[(iv)] $\pi_T(\textnormal{Id})=T$; %, and $\pi_T\left((t\pm i)^{-1}\right)=(D\pm i)^{-1}$;
		\item[(v)] $\|\pi_T(f)\|_{\calL_B(\cM)}\leq\sup_{\lambda \in \textnormal{spec}_{\cM}(T)}|f(\lambda)|$;
%		\item If for some $f\in C_b(\mathbb{R})$, $F\in C(\mathbb{R})$ is given by $F(t)=t\cdot f(t)$, then $\textnormal{dom}(T)\subseteq\textnormal{dom}(F(T))$, and for all $x\in\textnormal{dom}(T)$ we have $F(T)x=Tf(T)x=f(T)Tx$. If $F$ is bounded, then $\im(f(T))\subseteq\textnormal{dom}(T)$, and we have $F(T)=Tf(T)\in\mathcal{L}_B(\mathcal{M})$.
	\end{enumerate}
\end{theorem}
In the context of this theorem, we write $f(T) := \pi_T(f)$.

Suppose that there are a $G$-equivariant, Hermitian vector bundle $F \to M$, a differential operator $P\colon \Gamma^{\infty}(E) \to \Gamma^{\infty}(F)$, a $G$-equivariant vector bundle endomorphism $R$ of $E$, and a constant $c>0$ such that
\beq{eq D2 pos}
D^2 = P^*P + R,
\eeq
and $R \geq c^2$, fibrewise outside $Z$. (The use of $c^2$ instead of $c$ is a convention here, which implies that $D \geq c$ outside $Z$ in an appropriate sense.)

In this setting, and when $G$ is trivial but without assuming $Z$ to be compact, Roe \cite{Roe16} developed localised index theory with values in the $K$-theory of a reduced completion of $C^*_{\alg}(M; Z, L^2(E))$. We will use an equivariant version of this index theory for the maximal completion, in terms of admissible modules. The reason for using the maximal completion is that we then obtain an index in the $K$-theory of $C^*_{\max}(G)$, to which we can apply an integration map to recover the $G$-invariant index from \cite{Mathai13} as a special case, see Theorem \ref{thm invar index}. The construction of the localised index is based on the following analogue of Lemma 2.3 and Theorem 2.4 in \cite{Roe16}.
\begin{proposition} \label{prop Roe max}
If $f \in C_c(\R)$ is supported in $[-c,c]$,
 then
\[
f(D) \in A = \cK_A(A) \subset \calL_A(A).
\]
\end{proposition}
This proposition is proved in Subsection \ref{sec proof prop Roe max}

Let $b\colon \R \to \R$ be a continuous, increasing, odd function, such that $b(x) = \pm 1$ for all $x \in \R$ with $|x|\geq c$. Then $b^2-1$ has the property of the function $f$ in Proposition \ref{prop Roe max}. So, in particular,  $b(D) \in \calL_A(A)$ is invertible modulo $\cK_A(A)$, and hence has an index in
\[
 K_*(\cK_A(A)) = K_*(A).
\]
This index lies in even $K$-theory if $D$ is odd with respect to a $G$-invariant grading on $E$, and in odd $K$-theory otherwise. See for example Definition 3.2 in \cite{GHM1} for details.

Explicitly, consider the case where is $D$ odd with respect to a $G$-invariant grading $E = E_+ \oplus E_-$. 
Let $C^{\infty}_{\ker}(X; L^2(E))^G_{\loc}$ be the algebra of kernels in $C^{\infty}_{\ker}(X; L^2(E))^G$ supported near $Z$.
Let $b(D)_+$ be the restriction of $b(D)$ to kernels in $C^{\infty}_{\ker}(X; L^2(E))^G_{\loc}$ that are sections of $E_+ \otimes E^*$. Then $b(D)_+$ is invertible modulo $\cK_A(A)$, and its inverse is the restriction of $b(D)$ to $E_- \otimes E^*$. Hence this operator defines a class $[b(D)_+] \in K_1(\calL_A(A)/\cK_A(A))$, and the index of $b(D)$ is defined as
\beq{eq bdry index}
\partial[b(D)_+] \in K_0(A),
\eeq
where $\partial\colon K_1(\calL_A(A)/\cK_A(A)) \to K_0(\cK_A(A))$ is the boundary map in the six-term exact sequence correspondig to the ideal $\cK_A(A) \subset \calL_A(A)$. For ungraded operators, one uses the projection $\frac{1}{2}(b(D)+1)$ in  $\calL_A(A)/\cK_A(A)$ and applies the boundary map to its class in even $K$-theory to obtain the index of $b(D)$ in $K_1(A)$.
\begin{definition} \label{def loc index}
The \emph{localised, maximal, equivariant index} of $D$ is the image of the index of $b(D)$ in $K_*(A)$ described above under the map
\[
\oplus\,0\colon K_*(A) \to K_*(C^*_{\max}(G)).
\]
It is denoted by $\ind_G^{\loc}(D)$.
\end{definition}

\begin{remark}
One could consider \eqref{eq bdry index} (and its analogue in $K_1(A)$ in the non-graded case) as a localised index of $D$, defined in terms of the non-admissible $C_0(M)$-module $L^2(E)$. Two advantages of the index in Definition \ref{def loc index} over  \eqref{eq bdry index} are that it takes values in a $K$-theory group independent of $X$ or $E$, and that the application of the map $\oplus\,0$ on $K$-theory means that the index of Definition \ref{def loc index} captures group-theoretic information that is not encoded in \eqref{eq bdry index}. This is clear in the example where $G$ is compact,  $M$ is a point, and $D$ is the zero operator on $E = V \in \hat K$, as discussed in Example 3.8 in \cite{GHM1}. This illustrates why it is useful to use the admissible $C_0(M)$-module $L^2(E) \otimes L^2(G)$.

Another approach to constructing the index of Definition \ref{def loc index} would be to use an extension map $\oplus\,1$, extending operators by the identity operator on the orthogonal complement to $j(L^2(E))$, before applying boundary maps. See Definition 3.6 and Lemma 3.7 in \cite{GHM1}.
\end{remark}

\begin{example}
If $D$ is a Dirac-type operator associated to a Clifford connection $\nabla$ on $E$, then
\[
D^2 = \nabla^*\nabla + R,
\]
for a vector bundle endomorphism $R$ of $E$. (If $D$ is a $\Spin$-Dirac operator, then $R$ is scalar multiplication by a quarter of scalar curvature, by Lichnerowicz' formula.) If $R \geq c^2$ outside $Z$, then the condition on $D$ holds, with $F = E \otimes T^*M$ and $P = \nabla$. This is the situation considered in \cite{Roe16}, for $G$ trivial.
\end{example}

\begin{example} \label{ex Callias}
Let $\tilde D$ be a $G$-equivariant Dirac operator on $E$, and let $\Phi$ be a $G$-equivariant vector bundle endomorphism of $E$. Suppose that $\{\tilde D, \Phi \} := \tilde D\Phi + \Phi \tilde D$ is a vector bundle endomorphism of $E$, and that 
\beq{eq Callias est}
\{\tilde D, \Phi \} + \Phi^2 \geq c^2
\eeq
fibrewise outside $Z$. Then $D := \tilde D + \Phi$ satisfies the conditions on $D$ as above, with $F = E$, $P = \tilde D$ and $R = \{\tilde D, \Phi \} + \Phi^2$. 

This type of operator is a \emph{Callias-type operator}. Indices of Callias-type operators equivariant under proper actions were studied in \cite{Guo18, GHM1}.
\end{example}
The main application of the maximal localised index in this paper, Theorem \ref{thm QR=0}, is about the maximal localised index of Callias-type operators.

\subsection{The  invariant index} \label{sec invar index}

Integrating $L^1$-functions over $G$ extends to a trace on $C^*_{\max}(G)$. We will see in Theorem \ref{thm invar index} that applying this trace to the localised index of $D$ recovers an index defined in terms of $G$-invariant sections in \cite{Mathai13}. This fact will be used
 in the proof of Theorem \ref{thm QR=0}. It can also be used to obtain refined index theoretic information on non-compact manifolds; see Remark \ref{rem invar refine}.

Let $\chi \in C^{\infty}(M)$ a function with the property \eqref{eq chi cutoff}.
Consider the space $\Gamma_{\tc}(E)^G$ of transversally compactly supported sections of $E$, defined as the space of continuous, $G$-invariant sections  of $E$ whose supports have compact images in $M/G$ under the quotient map. The Hilbert space $L^2_T(E)^G$ of \emph{$G$-invariant, transversally $L^2$-sections of $E$} is the completion of $\Gamma_{\tc}(E)^G$ in the inner product
\[
(s_1, s_2)_{L^2_T(E)^G} := (\chi s_1, \chi s_2)_{L^2(E)}.
\]
The space $L^2_T(E)^G$ is independent of the choice of $\chi$; see Lemma 4.4 in \cite{Mathai13}.

Suppose that $D$ is odd with respect to a $G$-invariant grading $E = E_+ \oplus E_-$.
In Proposition 4.7 in \cite{Mathai13}, it is shown that $D$ defines a Fredholm operator $\bar D$ from a suitable Sobolev space inside $L^2_T(E)^G$ into $L^2_T(E)^G$. In Proposition 4.8 in the same paper, it is deduced that the space
\[
\ker(D) \cap L^2_T(E)^G
\]
is finite-dimensional, and that the index of $\bar D$ equals
\beq{eq def invar index}
\dim\bigl(\ker(D) \cap L^2_T(E_+)^G) - \dim\bigl(\ker(D) \cap L^2_T(E_-)^G).
\eeq
\begin{definition} \label{def invar index}
The \emph{$G$-invariant index} of $D$, denoted by $\ind(D)^G$, is the number \eqref{eq def invar index}. 
\end{definition}
In \cite{Braverman14}, Braverman further develops the theory of this index, when applied to Dirac operators with an added zero-order term that is relevant to geometric quantisation, and in particular proves that it is invariant under a suitable notion of cobordism. 

The map from $L^1(G)$ to $\C$ given by integrating functions over $G$ extends continuously to a $*$-homomorphism, or a trace
\[
I \colon C^*_{\max}(G) \to \C.
\]
The integer
\[
I_*(\ind_G^{\loc}(D) ) \in K_0(\C) = \Z
\]
plays the role of the $G$-invariant part of the localised  index of $D$, and  this will be made precise in Theorem \ref{thm invar index} below. 

If $M/G$ is compact, then all smooth sections of $E$ are transversally $L^2$. Then
the $G$-invariant index of $D$ equals 
\[
\dim\bigl(\ker(D_+)^G) - \dim\bigl(\ker(D_-)^G),
\]
where $D_{\pm}$ is the restriction of $D$ to sections of $E_{\pm}$. This index was developed and applied by Mathai and Zhang in \cite{Mathai10}, with an appendix by Bunke. In Theorem 2.7 and Proposition D.3 in that paper, 
it is shown
that the index can be recovered from the equivariant index of $D$ in $K_0(C^*_{\max}(G))$, defined via the analytic assembly map, if one applies the integration trace $I$. We will show that this generalises to the index in Definition \ref{def loc index} in the non-cocompact case.
%
%The integration trace just mentioned is the map
%\beq{eq int trace}
%I \colon C^*_{\max}(G) \to \C
%\eeq
%which is the continuous extension of the map from $L^1(G)$ to $\C$ mapping $f \in L^1(G)$ to $\int_Gf(g)\, dg$. 
\begin{theorem} \label{thm invar index}
We have
\beq{eq invar index}
I_*(\ind_G^{\loc}(D)) = \ind(D)^G \quad \in \Z.
\eeq
\end{theorem}
This theorem is proved in Section \ref{sec pf invar index}.

\begin{remark} \label{rem invar refine}
Theorem \ref{thm invar index} allows us to construct a more refined invariant of  
 operators that are invertible at infinity in the non-equivariant case than their Fredholm index. Suppose that $M$ is the universal cover of a manifold $\bar M$ and that $G = \Gamma$ is the fundamental group of $\bar M$, acting on $M$ in the natural way. Let $\bar D$ be an elliptic operator on $\bar M$ that is invertible at infinity in the appropriate sense, so that it lifts to a $\Gamma$-equivariant  operator on $M$ satisfying the conditions of Theorem \ref{thm invar index}. That theorem then implies that
 \beq{eq ind bar D}
 I_*(\ind_\Gamma^{\loc}(D)) = \ind(\bar D).
 \eeq
 In this sense,  $\ind_{\Gamma}^{\loc}(D)$ refines the Fredholm index of $\bar D$, much like the image of $D$ under the analytic assembly map for the maximal group $C^*$-algebra refines the Fredholm index of $\bar D$ if $\bar M$ is compact. 
  This can for example  be used to obtain stronger obstructions to Riemannian metrics of positive scalar curvature that classical obstructions from Callias index theory \cite{Anghel93} or the Gromov--Lawson index \cite{Gromov83}. 
 
 In the compact case, one can use the assembly map for the reduced group $C^*$-algebra here, and use the von Neumann trace to recover the index of $\bar D$ by Atiyah's $L^2$-index theorem \cite{AtiyahL2}. This is only possible because the action is free and the group is discrete. In the noncompact case, an analogue of this for the reduced localised index is not yet available. Moreover, Theorem \ref{thm invar index} does not rely on discreteness of the group acting or freeness of the action, which leads to 
 %generalisations of \eqref{eq ind bar D} that yield
  refined obstructions to positive scalar curvature on possibly noncompact orbifolds (refining Kawasaki's orbifold index, in the compact case) and more generally to metrics of positive scalar curvature invariant under a proper group action (refining the invariant index of Definition \ref{def invar index}).

Concrete applications to positive scalar curvature will be explored in future work.
 %
% More generally, for any proper, isometric action by a discrete group $\Gamma$ on $M$, the quotient $\tilde M = M/\Gamma$ is an orbifold, and the  $\ind^{\loc}_{\Gamma}(D)$ refines  an orbifold version of the index of $\tilde D$.
\end{remark}

Proposition 2.4 in \cite{HSIII} shows that the index defined in \cite{HS1} is another refinement of the invariant index. That index applies to Dirac operators with certain deformation terms added that are relevant to geometric quantisation. It takes values in the completion of the representation ring of a maximal compact subgroup of the group acting.

\subsection{Callias quantisation commutes with reduction}

In \cite{Mathai13, HM14}, the \emph{quantisation commutes with reduction} principle of Guillemin and Sternberg \cite{Guillemin82, Meinrenken98, Meinrenken99,  Paradan01, Zhang98} and its $\Spinc$-version \cite{Paradan17} is generalised to proper actions by possibly noncompact groups, with possibly noncompact orbit spaces, for suitably high powers of the prequantum or determinant line bundle in question. These results in \cite{Mathai13, HM14} are stated in terms of the invariant index of Definition \ref{def invar index}. The result in \cite{Mathai13} in the symplectic setting generalises the result in \cite{Mathai10} from compact to noncompact orbit spaces. This is generalised to the $\Spinc$-setting in \cite{HM14}. 

These were the first results on a version of the quantisation commutes with reduction principle where both the group and orbit space were allowed to be noncompact, but two drawbacks were that the invariant index used
\begin{enumerate}
\item was only well-defined after a deformation term (Clifford multiplication by the Kirwan vector field) was added to the relevant Dirac operator;
\item only involved $G$-invariant sections, and therefore provided no information about the parts of the kernel of $D$ on which $G$ acts nontrivially. 
\end{enumerate}
The second point was partially addressed in Theorem 2.13 in \cite{HSIII},  a quantisation commutes with reduction result for non-cocompact actions, where quantisation takes values in the completion of the representation ring of a maximal compact subgroup.

We are now able to remedy both points, in the case of Callias-type $\Spinc$-Dirac operators.

Let $D = D_{\Phi}$ be a Callias-type operator as in Example \ref{ex Callias}. Suppose that $E$ has a $G$-invariant $\Z/2$ grading, and that $\tilde D$ and $\Phi$, and hence $D$, are odd for this grading. Suppose that $E = \cS$ is the spinor bundle of a $G$-equivariant $\Spinc$-structure on $M$, and let $L \to M$ be its determinant line bundle. (The assumption that $E$ is $\Z/2$-graded now means that $M$ is even-dimensional.)
Suppose that
 $\tilde D$ is a $\Spinc$-Dirac operator on  $\cS$.  The Clifford connection on $\cS$ used to define $\tilde D$ can be constructed locally from a $G$-invariant, Hermitian connection $\nabla^L$ on $L$ and the connection on the spinor bundle for a local $\Spin$-structure; see e.g.\ Proposition D.11 in \cite{Lawson89}. This also induces a Clifford connection on the spinor bundle $\cS \otimes L^p$, for any $p \in \Z_{\geq 0}$. Let $\tilde D^{L^p}$ be the corresponding $\Spinc$-Dirac operator on $\cS \otimes L^p$. Set 
\[
D_p := \tilde D^{L^p}+\Phi \otimes 1_{L^p}.
\]
We have 
\beq{eq anticomm DLp Phi}
\{\tilde D^{L^p}, \Phi \otimes 1_{L^p}\} = \{\tilde D, \Phi\} \otimes 1_{L^p},
\eeq
where $\{\relbar, \relbar\}$ denotes the anticommutator. In what follows, we will omit `$\otimes 1_{L^p}$' from the notation. By \eqref{eq Callias est} and \eqref{eq anticomm DLp Phi},
we have 
\beq{eq Spinc Callias est}
\{\tilde D^{L^p}, \Phi\} + \Phi^2\geq c^2
\eeq
outside  $Z$, for all $p$. 
Hence $D_p$ has an index
\beq{eq ind Dp}
\ind_G^{\loc}(D_p) \in K_0(C^*_{\max}(G)).
\eeq

The quantisation commutes with reduction principle in general is an equality between the invariant part of the equivariant index of a $\Spinc$-Dirac operator and the index of a Dirac operator localised at the level set of a \emph{moment map}. The invariant part of the index will now be represented by the image of \eqref{eq ind Dp} under the integration trace $I$.

The \emph{$\Spinc$-moment map} associated to $\nabla^L$ is the map
\[
\mu\colon M \to \kg^*
\]
such that for all $X \in \kg$,
\[
2\pi i \langle \mu, X \rangle = \calL_X - \nabla^L_{X^M}, \quad \in \End(L) = C^{\infty}(M, \C),
\]
where $ \calL_X $ denotes the Lie derivative with respect to $X$, and $X^M$ is the vector field induced by $X$. Our sign convention is that for all $X \in \kg$ and $m \in M$,
\beq{eq def XM}
X^M(m) = \ddt \exp(-tX)\cdot m. 
\eeq

Suppose that $0$ is a regular value of $\mu$, and that $G$ acts freely on $\mu^{-1}(0)$. Suppose that the \emph{reduced space}
\[
M_0 := \mu^{-1}(0)/G
\]
is compact. In Lemma 3.3 in \cite{HM14}, a condition is given for $M_0$ to inherit a $\Spinc$-structure from $M$, with determinant line bundle $L_0^p$, with $L_0 = (L|_{\mu^{-1}(0)})/G \to M_0$. This is true for example if $G$ is semisimple, see Proposition 3.5 and Example 3.6 in \cite{HM14}. It is also true in the symplectic setting, where the $\Spinc$-structure on $M$ is associated to a $G$-invariant almost complex structure compatible with a $G$-invariant symplectic form, together with a $G$-equivariant, Hermitian line bundle on $M$. From now on, we assume such a $\Spinc$-structure on $M_0$ exists. Let $D_{M_0}^{L_0^p}$ be a $\Spinc$-Dirac operator on $M_0$ for this $\Spinc$-structure.
\begin{theorem}[Callias $\Spinc$-quantisation commutes with reduction] \label{thm QR=0}
There is a $p_0 \in \Z_{\geq 0}$ such that for all $p\geq p_0$,
\[
I_*(\ind_G(D_p)) = \ind(D_{M_0}^{L_0^p}) = \int_{M_0} \hat A(M_0)e^{\frac{p}{2}c(L_0)}.
\]
\end{theorem}
This theorem is proved in Section \ref{sec pf QR=0 1}.
%As a consequence of this result and Theorem \ref{thm Callias}, we find that the index from \cite{Guo18} also satisfies the quantisation commutes with reduction principle.
%\begin{corollary}
%There is a $p_0 \in \Z_{\geq 0}$ such that for all $p\geq p_0$,
%\[
%I_*(\ind_G^C(D_p)) = \ \int_{M_0} \hat A(M_0)e^{\frac{p}{2}c(L_0)}.
%\]
%Here $\ind_G^C$ is the index of Definition \ref{def Callias index}.
%\end{corollary}

\begin{remark}
In the case of Callias-type operators $D = \tilde D + \Phi$, as in Example \ref{ex Callias}, an index in $K_0(C^*(G))$ was constructed directly in \cite{Guo18}. Here $C^*(G)$ can be either the reduced or maximal group $C^*$-algebra. Let us denote this index by
\[
\ind_G^C(\tilde D + \Phi) \quad \in K_0(C^*(G)).
\]
Theorem 4.2 in \cite{GHM1} states that, for the reduced group $C^*$-algebra and Roe algebra, this index of Callias-type operators is a special case of the localised index:
\[
\ind_G^C(\tilde D + \Phi) = \ind_G^{\loc}(\tilde D + \Phi) \quad \in K_0(C^*_{\red}(G)).
\]
Via analogous arguments, one can show that this equality still holds for the maximal  group $C^*$-algebra and Roe algebra. Then Theorem \ref{thm QR=0} implies that, under the conditions in that theorem,
\[
I_*(\ind_G^C(\tilde D_p + \Phi)) = \ind(D_{M_0}^{L_0^p}). 
\]
%(\Todo: Hao, is all of this correct/plausible?)
\end{remark}

\begin{remark}
As far as we are aware, 
Theorem \ref{thm QR=0} was not known in the case where $G$ is compact, so that $D$ is Fredholm in the classical sense. In that case, by the standard shifting trick (see for example Corollary 1.2 in \cite{Meinrenken98}, and \cite{Paradan17}), Theorem \ref{thm QR=0} implies expressions for the multiplicities of all irreducible representations of $G$ in the equivariant index of $D_p$.
%, and therefore a complete decomposition of that index into irreducibles. 
One can handle cases where a reduced space $\mu^{-1}(\Ad^*(G)\xi)/G$ is not smooth by
\begin{enumerate}
\item using orbifold line bundles and indices if $\xi$ is a regular value of $\mu$ in an appropriate sense;
\item using reduced spaces at nearby regular values if $\xi$ is a singular value of $\mu$, as in \cite{Meinrenken99, Paradan01, Paradan17}. An alternative approach is developed in \cite{Loizides19}.
\end{enumerate}
\end{remark}

\section{Regularity and localisation} \label{sec reg loc}

In this section, we prove Theorem \ref{thm D reg} and Proposition \ref{prop Roe max}, which imply that the localised maximal index of Definition \ref{def loc index} is well-defined.

%%%%%%%%%%%
	
\subsection{Unbounded operators on maximal operator modules}
\label{sec unbdd ops}
	
Our first goal is to make sense of $D$ as an unbounded, regular, essentially self-adjoint operator on certain maximal operator modules that we now introduce.

	Let $M_1$ and $M_2$ be Riemannian manifolds equipped with proper, isometric $G$-actions. Let $E_1$ and $E_2$ be Hermitian $G$-vector bundles over $M_1$ and $M_2$ respectively. Consider the vector bundle $\Hom(E_2, E_1) := E_1 \boxtimes E_2^* \to M_1 \times M_2$.
	\begin{definition}
Denote by
\[
\HH_{\ker}^{\infty}(E_1,E_2)^G_{\loc}  \subset \Gamma^{\infty}(\Hom(E_2, E_1))
\]
	the $\mathbb{C}$-vector space of smooth $G$-invariant, cocompactly supported, finite propagation kernels. Here we say that a kernel $\kappa$ has cocompact support if there exists cocompact subsets $U_1$ and $U_2$ of $M_1$ and $M_2$ respectively such that $\supp(\kappa)\subseteq U_1\times U_2$. 
	\end{definition}
There is no natural product or $*$-operation on $\HH_{\ker}^{\infty}(E_1,E_2)^G_{\loc} $. However, it admits a natural action of $C^{\infty}_{\ker}(M_2;L^2(E_2))^G_{\loc}$ from the right, given by composition of kernels. Further, it has a $C^{\infty}_{\ker}(M_2;L^2(E_2))^G_{\loc}$-valued inner product given by
	$$\langle \kappa,\kappa'\rangle:=\kappa^*\kappa',$$
	for $\kappa, \kappa' \in \HH_{\ker}^{\infty}(E_1,E_2)^G_{\loc}$,
	defined through the usual adjoint and multiplication of kernels. This makes $\HH_{\ker}^{\infty}(E_1,E_2)^G_{\loc} $ a pre-Hilbert $C^{\infty}_{\ker}(M_2;L^2(E_2))^G_{\loc}$-module.
	
	Now taking the simultaneous completions of $\HH_{\ker}^{\infty}(E_1,E_2)^G_{\loc} $ and $C^{\infty}_{\ker}(M_2;L^2(E_2))^G_{\loc}$ (see p.\ 5 of \cite{Lance}) gives a Hilbert $C^*_{\max}(M_2;L^2(E))^G_{\loc}$-module that we denote by
	$\HH_{\max}(E_1,E_2)^G_{\loc} $.

%		\begin{definition}
%		Let $Z\subset M$ be a cocompact subset. Let $C_{\ker}^{\infty}(M;L^2(E))^G_{\loc}$ be the $*$-algebra of $G$-invariant, smooth Schwartz kernels on $L^2(E)$ with finite propagation and cocompact support in $M\times M$. Let $C^*_{\max}(M;L^2(E))^G_{\loc}$ denote the completion of $C_{\ker}^{\infty}(M;L^2(E))^G_{\loc}$
%		with respect to the norm of $C^*_{\max}(M;Z)^G$.
%	\end{definition}
%	\begin{remark}
%		In other words, $C^*_{\max}(M;L^2(E))^G_{\loc}$ is the completion of $C_{\ker}^{\infty}(M;L^2(E))^G_{\loc}$ with respect to the supremum norm over all $*$-representations of $$C^{\infty}_{\ker}(M;Z)^G\cong C_c^\infty(G)\otimes\mathcal{K}$$ restricted to the $*$-subalgebra $C_{\ker}^{\infty}(M;L^2(E))^G_{\loc}$.
%	\end{remark}

%(\Todo: This definition does not agree with Definition \ref{def max loc Roe}, because there kernels are not assumed to be supported cocompactly. What do we need to change?)	
	
	%It is clear that 
	In the case $M_1=M_2 = M$ and $E_1=E_2 = E$, equipped with the same $G$-action, then $\HH_{\ker}^{\infty}(E,E)^G_{\loc} =C_{\ker}^{\infty}(M;L^2(E))^G_{\loc}$. In this case $\HH_{\max}(E,E)^G_{\loc} $ is the $G$-maximal equivariant localised Roe algebra  $C^*_{\max}(M;L^2(E))^G_{\loc}$ of Definition \ref{def max loc Roe}.

To introduce the idea of an unbounded operator on these operator modules, let us first consider $C^*_{\max}(M;L^2(E))^G_{\loc}$ as a right Hilbert module over itself, as in Subsection \ref{sec loc index}.
	We will consider $D$ as an unbounded, densely defined operator on this Hilbert module. Note that $D$ acts naturally on smooth Schwartz kernels via differentiation on the first coordinate, so it defines a map
	$$D:C_{\ker}^{\infty}(M;L^2(E))^G_{\loc}\rightarrow C_{\ker}^{\infty}(M;L^2(E))^G_{\loc}.$$
%	In keeping with the usual convention for operators on Hilbert modules, we would like the domain of $D^l$ to be a right $C^*_{\max}(M;L^2(E))^G_{\loc}$-module. (\Todo: why is this important now?) Thus denote the unitisation by $(C^*_{\max}(M;L^2(E))^G_{\loc})^+$. Then the right ideal $$C_{\ker}^{\infty}(M;L^2(E))^G_{\loc}\cdot(C^*_{\max}(M;L^2(E))^G_{\loc})^+$$ 
%	contains $C_{\ker}^{\infty}(M;L^2(E))^G_{\loc}$. Further, the operator $D^l$ is defined on this domain, by the property that
%(\Todo: I assume this is what was meant here, the old text was ``it admits an action from the right by $C^*_{\max}(M;L^2(E))^G_{\loc}$ given,'')
%	  for each $a \in C_{\ker}^{\infty}(M;L^2(E))^G_{\loc}$ and each $b\in C^*_{\max}(M;L^2(E))^G_{\loc}$ by
%	$$D^l(a\cdot b):=\lim_{n\rightarrow\infty}D^l a\cdot b_n=(D^l a)\cdot b,$$
%	where $b_n$ is a sequence in $a\in C_{\ker}^{\infty}(M;L^2(E))^G_{\loc}$ converging to $b$.
%	
%	Thus we obtain a family of densely defined, unbounded operators $D^l$ on the desired Hilbert module. Note that the action on $D^l$ on its domain is determined by its action on $C_{\ker}^{\infty}(M;L^2(E))^G_{\loc}$, so that in practice we may work with $C_{\ker}^{\infty}(M;L^2(E))^G_{\loc}$ as the dense domain instead.
%	
%	Finally, symmetricity of $D$ on $L^2(E)$ implies that $D^l$ is symmetric with respect to the inner product \eqref{eq A inner prod A}. (\Todo: where do we use this?)
	
	More generally, we may consider the situation when $M_1$ and $M_2$ are proper $G$-manifolds, and $E_1,E_2$ are Hermitian $G$-vector bundles over $M_1$ and $M_2$ respectively. Let $D$ be a symmetric, $G$-equivariant differential operator acting on a $E_1$.
	Then $D$ defines an unbounded, symmetric operator 
	$$D:\HH_{\max}(E_1,E_2)^G_{\loc} \rightarrow\HH_{\max}(E_1,E_2)^G_{\loc} $$ 
	with initial domain $\HH_{\ker}^{\infty}(E_1,E_2)^G_{\loc} $.

%%%%%%%%%
		
\subsection{Kernels on $M \times M$ and on $G \times G$}

Via the isomorphism \eqref{eq geom adm}, operators on $L^2(E)$ with smooth kernels define kernels on $G\times G$, in a way that we make explicit in Lemma \ref{lem MM GG}. Using that lemma, we obtain  estimates for certain kernels on $G \times G$ in Subsection \ref{sec est GG}, which are then used to prove the case of Theorem \ref{thm D reg} for cocompact actions, Proposition \ref{prop:cocompactregularity}, in Subsection \ref{sec reg cocpt}.

In order to make estimates in the norm on $C^*_{\max}(M;L^2(E))^G$, we need to work with the maximal norm on the Roe algebra of the corresponding admissible module, namely $L^2(E)\otimes L^2(G)$. Thus let $j$ and $p$ be the inclusion and projection maps defined in Subsection \ref{sec plus 0}, and consider the map $\oplus\,0$ as in \eqref{eq def plus 0 B}.
%	$$j:L^2(E)\hookrightarrow L^2(E)\otimes L^2(G),$$
%	denote the equivariant inclusion given by $j(s)(x,g)=\chi(g^{-1}x)s(x).$ Let
%	$$p:L^2(E)\otimes L^2(G)\rightarrow j(L^2(E))$$
%	denote the orthogonal projection onto the image of $j$. 
	Let 
	\beq{eq Psi}
\Psi:L^2(G)\otimes \HH\rightarrow L^2(E)\otimes L^2(G)
\eeq
	be the $G$-equivariant unitary isomorphism as in \eqref{eq geom adm}, as defined in (21) in \cite{GHM1} and reviewed below. 

We assume for the rest of this subsection that	$M = G\times_K N$ for a single slice $N \subset M$.	
	We then have $\HH = L^2(K\backslash G)\otimes L^2(E|_N)$, and  the isomorphism $\Psi$ is pullback along a $G$-equivariant, measure preserving bijection
	\[
	\psi\colon M \times G \to  G \times K \backslash G \times N .
	\]
We will use the explicit form of this bijection from Lemma 5.2 in \cite{GHM1}: for a fixed Borel section $\phi\colon K \backslash G \to G$, and for $g,h \in G$ and $y \in N$,
\[
\psi(gy, n) = \bigl( h\phi(Kg^{-1}h)^{-1}, Kg^{-1}h, \phi(Kg^{-1}h)h^{-1}gy\bigr).
\]	
We will use an expression for the inverse of this map.
\begin{lemma}\label{lem psi inv}
There is a measurable map $\eta\colon G \times K \backslash G \times N \to K$ such that for all $x_1, x_2 \in G$ and $n \in N$,
\[
\psi^{-1}(x_1, Kx_2, n) = \bigl( x_1n, x_1 \eta(x_1, Kx_2, n)x_2 \bigr).
\]
\end{lemma}
\begin{proof}
Let $x_1, x_2, g, h \in G$ and $n, y \in N$, and suppose that 
\[
\psi(gy, h) = (x_1, Kx_2, n). 
\]

Then the elements
\[
\begin{split}
k&:= \phi(Kg^{-1}h)h^{-1}g\\
k'&:= x_2 h^{-1}g
\end{split}
\]
lie in $K$. They depend measurably on $g$, $h$ and $x_2$, and hence on $x_1$, $Kx_2$ and $n$ because $\psi$ and its inverse are measurable. One then finds that $gy = x_1n$ and $h = x_1 kk'^{-1}x_2$, so the claim follows.
\end{proof}
	
%	Let $\oplus\,0$ denote the map
%	$$\mathcal{B}(L^2(E))\hookrightarrow\mathcal{B}(L^2(E)\otimes L^2(G))$$
%	defined in \cite{GHM1} subsection 3.4\footnote{There the map $\oplus\,0$ was defined on the reduced Roe algebra, but there is an obvious extension to $\mathcal{B}(L^2(E))$.}. 
For any $T\in\mathcal{B}(L^2(E))$, the operator $\tilde{T}$ on $L^2(G) \otimes \HH$ corresponding to $T\,\oplus 0$ can be written as 
	\beq{eq T  tilde}
	\tilde{T}:=\Psi^{-1}(T\oplus 0)\Psi=\Psi^{-1}\circ j\circ T\circ j^{-1}\circ p\circ\Psi.
	\eeq
We use the same notation for Schwartz kernels of such operators.
\begin{lemma}\label{lem MM GG}
Let $\eta$ be as in Lemma \ref{lem psi inv}. Then for all $\kappa \in \Gamma(\Hom(E))^G$ defining a bounded operator on $L^2(E)$, and all $g,g',h,h' \in G$ and $y,y' \in N$,
\[
\tilde \kappa(g, Kh, y; g', Kh', y') = \chi(h^{-1}\eta(g, Kh, y)y) \chi(\phi(Kh')^{-1}y') \kappa(gy, g'y').
\]
\end{lemma}
\begin{proof}
By (4.15) in \cite{HWW} (or a direct check), we have for all $\zeta \in L^2(E) \otimes L^2(G)$ and $m \in M$,
\[
(j^{-1} \circ p)(\zeta)(m) = \int_G \chi(g^{-1}m)\zeta(m,g)\, dg.
\]
Using this equality, the definition $\Psi = \psi^*$, the definition \eqref{eq def j} of $j$ and the definition \eqref{eq T tilde} of the tilde operation, we find that for all $\varphi \in L^2(G) \otimes L^2(K \backslash G) \otimes L^2(E|_N)$ and all $g,h \in G$ and $y \in N$,
\beq{eq tilde kappa phi}
(\tilde \kappa \varphi)(\psi(gy, h)) = 
\chi(h^{-1}gy) \int_G \int_G \int_N
\chi(g''^{-1} g'y') \kappa(gy, g'y') \varphi(\psi(g'y', g''))\, dy'\, dg'\, dg''.
\eeq
Substituting
\[
\begin{split}
g_1 &= g'^{-1}g'' \quad \text{for $g'$},\\
g_2 &= g'' \phi(Kg'^{-1}g'')^{-1} \quad \text{for $g''$, and }\\
y_1 &= \phi(Kg'^{-1}g'') g''^{-1}g' y' \quad \text{for $y'$},
\end{split}
\]
and using unimodularity of $G$,
we compute that the right hand side of \eqref{eq tilde kappa phi} equals
\[
\chi(h^{-1}gy) \int_G \int_G \int_N
\chi(\phi(Kg_1)^{-1}y_1) \kappa(gy, g_2 y_1) \varphi(g_2, Kg_1, y_1)\, dy_1\, dg_1\, dg_2.
\]
So for all $g,h, g_1, g_2 \in G$ and $y,y_1 \in N$,
\[
\tilde \kappa(\psi(gy, h); g_2, Kg_1, y_1) = 
\chi(h^{-1}gy)\chi(\phi(Kg_1)^{-1}y_1) \kappa(gy, g_2 y_1). 
\]
The claim now follows from Lemma \ref{lem psi inv}.
\end{proof}

\subsection{Estimates for kernels on $G\times G$}	\label{sec est GG}

We will use Lemma \ref{lem MM GG} to obtain estimates (Lemmas \ref{lem L1G2infty} and \ref{lem sigma kappa}) for certain kernels $\tilde \kappa$ that are used in the proof of Proposition \ref{prop:cocompactregularity}. In Lemma \ref{lem MM GG}, it was assumed that $M = G \times_K N$ for a single slice $N$. That is true for almost connected $G$ \cite{Abels}, but for general cocompact $M$, one needs a finite set of slices \cite{Palais61}. We state Lemmas \ref{lem L1G2infty} and \ref{lem sigma kappa} in that more general setting. In their proofs, we consider the case of a single slice first, and discuss how that implies the more general case.

Suppose that $M/G$ is compact. Then by Palais' slice theorem \cite{Palais61}, there are finitely many compact subgroups $K_j < G$ and relatively compact submanifolds $N_j \subset M$ such that 
\begin{enumerate}
\item
for each $j$, the open set $U_j := G \cdot N_j$ is $G$-equivariantly diffeomorphic to $G \times_{K_j} N_j$, 
\item the union of the closures of the sets $U_j$ is all of $M$, and
\item the closures of any two of the sets $U_j$ intersect in a set of measure zero.
\end{enumerate}
Then the Hilbert space $\HH$ in \eqref{eq Psi} may be taken to be
\[
\HH = \bigoplus_j L^2(K_j \backslash G) \otimes L^2(E|_{N_j}),
\]
and the map $\Psi$ in \eqref{eq Psi} is defined in terms of the decomposition $L^2(E) = \bigoplus_j L^2(E|_{U_j})$, and maps
\[
\Psi_j\colon L^2(G) \otimes L^2(E|_{N_j}) \to L^2(E|_{U_j}) \otimes L^2(G)
\]
defined as in the case of a single slice. 

As before, a $G$-equivariant bounded operator defines operators $T \oplus 0$ on $L^2(E) \otimes L^2(G)$ and $\tilde T = \Psi^{-1}(T \oplus 0) \Psi$ on $L^2(G) \otimes \HH$. If such an operator $T$ is defined by a smooth (or measurable) kernel $\kappa \in \Gamma(\Hom(E))^G$, then $\tilde T$ is defined by a measurable kernel $\tilde \kappa$. For all $g,g' \in G$, the operator $\tilde \kappa(g,g')$ on $\HH$ is given by a kernel
\[
\tilde \kappa(g,g') \in \bigoplus_{j,k} \Gamma(\Hom(K_j\backslash G \times E|_{N_j}, K_k\backslash G \times E|_{N_k}) ).
\]
Its component $\tilde \kappa_{j,k}(g,g')$ in $\Gamma(\Hom(K_j\backslash G \times E|_{N_j}, K_k\backslash G \times E|_{N_k}) )$ is a section of the vector bundle
\[
\Hom(K_j\backslash G \times E|_{N_j}, K_k\backslash G \times E|_{N_k})  \to 
(K_j\backslash G \times N_j) \times (K_k\backslash G \times N_k).
\]
Whenever the following expressions converge, we set
\[
\begin{split}
\| \tilde \kappa_{j,k}(g,g') \|_{2,
\infty} &:= \sup_{(K_k h', y') \in K_k \backslash G \times N_k}
\left(
\int_{K_j \backslash G \times N_j} \|\tilde \kappa_{j,k}(g,g')(K_j h, y; K_k h', y')\|^2
d(K_j h)\, dy
 \right)^{1/2};\\
 \| \tilde \kappa \|_{L^1(G), 2,
\infty} &:= \sum_{j,k} \int_G \|\tilde \kappa_{j,k}(e,g)\|_{2,\infty}\, dg.
\end{split}
\]
The norm in the integrand in the top line is the operator norm on $\Hom(E_{y'}, E_{y})$.

\begin{lemma}\label{lem L1G2infty} 
Suppose that $\kappa_0 \in \Gamma^{\infty}(\Hom(E))^G$ has finite propagation and defines a bounded operator on $L^2(E)$. For $\mu >0 $, let $\kappa$ be the kernel  of the composition $(D+i\mu)^{-1} \circ \kappa_0$ of bounded operators on $L^2(E)$. Then for $\mu$ large enough, the norm $ \| \tilde \kappa \|_{L^1(G), 2,
\infty} $ is well-defined and finite.
\end{lemma}

Let $\sigma\colon M \times M \to \End(E)$ be bounded, and $G$-invariant in the sense that for all $g \in G$ and $m,m' \in M$,
\[
\sigma(gm, gm') = g \sigma(m,m') g^{-1}.
\]
Then for all $\kappa \in \Gamma(\Hom(E))^G$, we define the section $\sigma\kappa \in \Gamma(\Hom(E))^G$ by
\beq{eq def sigma kappa}
(\sigma \kappa)(m,m') = \sigma(m,m')_{m} \circ \kappa(m,m'),
\eeq
for $m,m' \in M$. Here $\sigma(m,m')_{m}  \in \End(E_{m})$ is the value of $\sigma(m,m') \in \End(E)$ at $m$. We will identify the element $\widetilde{\sigma \kappa} \in \cB(L^2(G) \otimes \HH)$ with the map $G \to \cB(\HH)$ defined by $g \mapsto \widetilde{\sigma \kappa} (e,g)$.
\begin{lemma}\label{lem sigma kappa}
Suppose that $\kappa \in \Gamma(\Hom(E))^G$ is such that $ \| \tilde \kappa \|_{L^1(G), 2,
\infty} $ is finite, and suppose that $\sigma\colon M \times M \to \End(E)$ is bounded and $G$-invariant. Then $\widetilde{\sigma \kappa} \in L^1(G) \otimes \cB(\HH)$, and 
\[
\|\widetilde{\sigma \kappa} \|_{L^1(G) \otimes \cB(\HH)} \leq \|\sigma \|_{\infty} \| \tilde \kappa \|_{L^1(G), 2,
\infty}. 
\]
\end{lemma}

\subsection{Proofs of Lemmas \ref{lem L1G2infty} and \ref{lem sigma kappa}}

We first prove Lemma \ref{lem sigma kappa}. This is based on the following two lemmas. 
\begin{lemma} \label{lem B 2 infty}
Let $\tau \in \Gamma(\Hom(K_j\backslash G \times E|_{N_j}, K_k\backslash G \times E|_{N_k}) )$, and suppose that $\|\tau \|_{2,\infty}$ is finite. Then $\tau$ defines a bounded operator from $L^2(K_j \backslash G) \otimes L^2(E_{N_j})$ to $L^2(K_k\backslash G) \otimes L^2(E_{N_k})$, and its operator norm is at most equal to  $\|\tau \|_{2,\infty}$.
\end{lemma}
\begin{proof}
This can be checked directly, and is in fact a special case of a more general statement about operators defined by measurable kernels.
\end{proof}
\begin{lemma}\label{lem sigma kappa g g'}
Suppose that $\sigma\colon M \times M \to \End(E)$ is bounded and $G$-invariant, and let $\kappa \in \Gamma(\Hom(E))^G$. Let $g,g' \in G$, and suppose that $\| \tilde \kappa(g,g') \|_{2,\infty}$ is finite. Then 
\[
\|\widetilde{\sigma \kappa}(g,g')\|_{2,\infty} \leq \|\sigma\|_{\infty} \| \tilde \kappa(g,g') \|_{2,\infty}.
\]
\end{lemma}
\begin{proof}
This follows directly from Lemma \ref{lem MM GG} and the fact that $\|\chi\|_{\infty} \leq 1$.
\end{proof}
Lemma \ref{lem sigma kappa} follows from Lemmas \ref{lem B 2 infty} and \ref{lem sigma kappa g g'}.

The proof of Lemma \ref{lem L1G2infty} starts with an off-diagonal estimate for the kernel of $(D+i\mu)^{-1}$.
	\begin{lemma}
		\label{lem:kerneldecay}
Let $\mu>0$, and let $\kappa_{\mu}$ be the Schwartz kernel of $(D+i\mu)^{-1}$.		
		There exists a constant $C_\mu>0$ such that for all $m,m'\in M$ satisfying $d(m,m')\geq 1$, 	
\[
\norm{\kappa_{\mu}(m,m')}\leq C_{\mu}e^{-\frac{\mu}{2}d(m,m')},
\]
		where the norm is taken fibrewise in $E\boxtimes E^*$.
	\end{lemma}
	\begin{proof}
		This is Lemma 3.3 in \cite{GXY}. We remark that the proof works because $M$ and $E$ have bounded geometry.
	\end{proof}	
	
\begin{proof}[Proof of Lemma \ref{lem L1G2infty}.]
Let $\kappa$ be as in the lemma. 
Lemma \ref{lem:kerneldecay} and the fact that $\kappa_0$ has finite propagation imply that for all $\mu>0$, there is an $A_{\mu}>0$ such that for all $m, m' \in M$,
\beq{eq est kappa}
\|\kappa(m,m')\| \leq A_{\mu} e^{-\mu d(m,m')/2}.
\eeq

First suppose that $M = G \times_K N$ for a single slice $N$.
Because $M/G$ is compact, the slice $N$ and the support of $\chi$ are compact. So
by properness of the action by $G$ on $M$, the function
\[
 \chi_N\colon g \mapsto \int_N \chi(g^{-1}y)\, dy
\]
on $G$ has compact support. Thus $\phi^{-1}(\supp(\chi_N)) \subset K\backslash G$ has finite volume.

Using Lemma \ref{lem MM GG}, the estimate \eqref{eq est kappa}, and the fact that $\|\chi\|_{\infty} \leq 1$, we find that
\beq{eq tilde kappa L1G2infty}
\|\tilde \kappa\|_{L^1(G), 2, \infty} \leq A_{\mu}
\vol\bigl(\phi^{-1}(\supp(\chi_N))\bigr) \int_G \sup_{y' \in N}  
\left(\int_N e^{-\mu d(y,gy')} \, dy \right)^{1/2}\, dg.
\eeq
For fixed $y' \in N$, the map $g\mapsto gy'$ is a quasi-isometry from $G$ to $M$. Together with compactness of $N$, this implies that there are $a,b>0$ such that for all $g \in G$ and all $y' \in N$,
\[
\int_N e^{-\mu d(y,gy')} \, dy \leq ae^{-b\mu d_G(e,g)}.
\]
By the volume growth condition at the start of Section \ref{sec prelim}, we can choose $\mu$ large enough so that
 the right hand side of \eqref{eq tilde kappa L1G2infty} converges.
%
%By properness of the action and compactness of $N$, the set
%\[
%A:= \{g \in G;  d(y,gy') \leq 1 \text{for some $y,y' \in N$}\}
%\]
%is compact. Lemma \ref{lem:kerneldecay} implies that the right hand side of \eqref{eq tilde kappa L1G2infty} is at most equal to
%\[
%C_{\mu}
%\vol\bigl(\phi^{-1}(\supp(\chi_N))\bigr) \int_G \sup_{y' \in N}  
%\left(\int_N e^{-{\mu}d(m,m')} \, dy \right)^{1/2}\, dg\\
%\]
\end{proof}

\subsection{The cocompact case of Theorem \ref{thm D reg}}\label{sec reg cocpt}
	
%%%% Old %%%%
	
	\begin{proposition}
	\label{prop:cocompactregularity}
	Let $M$ be a cocompact $G$-manifold. Then $D$ is a regular and essentially self-adjoint operator on the Hilbert $C^*_{\max}(M;L^2(E))^G$-module $C^*_{\max}(M;L^2(E))^G$.		
	\end{proposition}
	\begin{proof}
Proposition 4.1 in \cite{KaadLesch} (which is based on Lemmas 9.7 and 9.8 in \cite{Lance}) states that $D$ is both essentially self-adjoint and regular if there is a $\mu>0$ such that $D +i\mu$ and $D-i\mu$ have dense ranges. We will find a $\mu>0$ such that  $D +i\mu$ has dense range; the argument for $D -i\mu$ is similar.
	
Let $\kappa_0\in C^{\infty}_{\ker}(M,L^2(E))^G$, and let $\mu > 0$. Consider the composition
\[
\kappa:=(D+i\mu)^{-1}\kappa_0
\] 
of bounded operators on $L^2(E)$. It is unclear a priori if $\kappa$ lies in the initial domain of $D+i\mu$ as an unbounded operator on the Hilbert	$C^*$-module $C^*_{\max}(M;L^2(E))^G$. To remedy this, we consider a family $\{f_{\varepsilon}\}_{\varepsilon \in (0,1]}$ of smooth functions on $M \times M$, invariant under the diagonal action by $G$, such that for all $\varepsilon \in (0,1]$,
\begin{enumerate}
\item $f_{\varepsilon} = 1$ on $\Pen(\supp(\kappa_0), 1/\varepsilon)$;
\item $f_{\varepsilon} = 0$ outside $\Pen(\supp(\kappa_0), 3/\varepsilon)$;
\item $\|df_{\varepsilon}\|_{\infty} \leq \varepsilon$.
\end{enumerate}
Because $M$ is complete, the sets $\Pen(\supp(\kappa_0), 3/\varepsilon)$ are cocompact (with respect to the diagonal action by $G$). So the functions $f_{\varepsilon}$ are cocompactly supported.

We write $d_1 f_{\varepsilon} \in \Gamma^{\infty}(M \times M, T^*M \times M)$ for the derivative of such a function $f_{\varepsilon}$ in the first factor. Then for all $\varepsilon \in (0,1]$, $f_\varepsilon\kappa$ lies in the domain of the Hilbert $C^*$-module operator $D+i\mu$, and 
\[
        (D+i\mu)(f_\varepsilon\kappa)=f_{\varepsilon}(D+i\mu)\kappa+\sigma_D(d_1 f_{\varepsilon})\kappa
        =\kappa_0+\sigma_D(d_1 f_{\varepsilon})\kappa.
\]
Here the composition $\sigma_D(d_1 f_{\varepsilon})\kappa$ is defined as in \eqref{eq def sigma kappa}.

By \eqref{eq max norm gp Cstar} and \eqref{eq max norm L2E},
% and because the norm on $C^*_{\max}(G)$ is bounded above by the norm on $L^1(G)$, 
we find that
\[
\begin{split}
\|(D+i\mu)(f_\varepsilon\kappa) - \kappa_0\|_{\max} &= 
\|\sigma_D(d_1 f_{\varepsilon})\kappa\|_{\max} \\
&=
\|\widetilde{\sigma_D(d_1 f_{\varepsilon})\kappa}\|_{C^*_{\max}(G) \otimes \cB(\HH)}\\
&\leq 
\|\widetilde{\sigma_D(d_1 f_{\varepsilon})\kappa}\|_{L^1(G) \otimes \cB(\HH)}.
\end{split}
\]
Lemmas \ref{lem L1G2infty} and \ref{lem sigma kappa} imply that, for $\mu$ large enough (independently of $\kappa_0$),  the right hand side is at most equal to
\beq{eq est sigma D}
\|\sigma_D(d_1 f_{\varepsilon})\|_{\infty} \|\tilde \kappa\|_{L^1(G), 2, \infty}.
\eeq
  Because $D$ has finite propagation speed, there is a $C>0$ such that  $\| \sigma_D(\xi) \| \leq C\|\xi\|$ for all $\xi \in T^*M$. So \eqref{eq est sigma D} is at most equal to
  \[
C \|d_1 f_{\varepsilon}\|_{\infty} \|\tilde \kappa\|_{L^1(G), 2, \infty}   \leq C  \varepsilon \|\tilde \kappa\|_{L^1(G), 2, \infty}.
  \]

We conclude that, for $\mu$ large enough, any element $\kappa_0\in C^{\infty}_{\ker}(M,L^2(E))^G$ can be approximated arbitrarily closely by an element in the image of $D+i\mu$. Since $C^{\infty}_{\ker}(M,L^2(E))^G$ is dense in $C^*_{\max}(M,L^2(E))^G$, it follows that $D+i\mu$ has dense range. By the first paragraph of this proof, the operator $D$ is regular and essentially self-adjoint.
\end{proof}
Proposition \ref{prop:cocompactregularity} generalises to kernels on the product of two different manifolds.
	\begin{proposition}
	\label{prop:arbitrarysecondmanifoldcocpt}
	Let $M_1$ and $M_2$ be proper $G$-manifolds with and $E_1,E_2$ $G$-vector bundles over $M_1$ and $M_2$ respectively. Suppose that $M_1$ is cocompact, and let $D$ be a symmetric, elliptic, first-order differential operator with finite propagation speed acting on sections of  $E_1$.
	Then the operator 
	$$D:\HH_{\max}(E_1,E_2)^G_{\loc} \rightarrow\HH_{\max}(E_1,E_2)^G_{\loc} $$
	is regular and essentially self-adjoint in the sense of Hilbert $C^*_{\max}(M_2, L^2(E_2))^G_{\loc}$-modules.
%	In fact, the proof of Theorem \ref{thm:cocompactregularity} still goes through if instead of $C^{\infty}_{\ker}(M,L^2(E))^G$ we considered the algebra of smooth $G$-invariant cocompactly supported kernels on $M_1\times M_2$, where $M_1$ is a $G$-cocompact manifold and $M_2$ is any proper $G$-manifold (together with appropriate bundles $E_1\rightarrow M_1$ and $E_2\rightarrow M_2$). Here, a cocompactly supported kernel $k$ is one for which there are $G$-cocompact subsets $U_1$ and $U_2$ of $M_1$ and $M_2$ respectively such that $\textnormal{supp}(k)\subseteq U_1\times U_2$.
	\end{proposition}
The proof of Proposition \ref{prop:cocompactregularity} can be applied  with minimal changes to prove Proposition \ref{prop:arbitrarysecondmanifoldcocpt}. We have only given the details for $M_1 = M_2$ for notational simplicity. The key point is that any kernel with finite propagation has cocompact support. Note that for this, it is only necessary that either $M_1$ or $M_2$ be cocompact.	
	
%	\begin{proof}
%	We want to show that, for a kernel $k_0\in\HH_{\ker}^{\infty}(E_1,E_2)^G_{\loc} $, we have that $k:=(D+i \mu)^{-1}k_0$ lies in $\HH_{\max}(E_1,E_2)^G_{\loc} $. For this, let $\{a_i\}_{i\in\mathbb{N}}$ be a family of smooth, cocompactly supported, $G$-invariant functions such that for each $i$, $\textnormal{supp}(a_i)$ is contained inside the ball of radius $i$ around the diagonal in $M_1\times M_1$...
%	family of $G$-invariant socompactly supported functions on $M_1$ as in the proof of Theorem \ref{thm:regularity}. Then one shows 
%	In other words, it suffices to show that $k'=k^*k\in C^*_{\max}(M_2;L^2(E))^G_{\loc}$.
%	
%	The fact that $k_0$ is cocompactly supported on $M_1\times M_2$ implies that $k'$ is cocompactly supported on $M_2\times M_2$.
%	
%	As in the proof of Theorem \ref{thm:cocompactregularity}, w
%	\end{proof}

\subsection{Proof of Theorem \ref{thm D reg}} \label{sec proof prop D reg}
%	\begin{theorem}
%	\label{thm:regularity}
%	Let $D$ be a symmetric first-order elliptic differential operator acting on a $G$-bundle $E$ over a complete $G$-manifold $M$. Then $D$ is an essentially self-adjoint, regular operator on $C^*_{\max}(M;L^2(E))^G_{\loc}$.
%	\end{theorem}
%	\begin{proof}

We will  prove Theorem \ref{thm D reg} using Proposition \ref{prop:cocompactregularity} for the cocompact case.

The following lemma will be used in a few places.
\begin{lemma} \label{lem vb end bdd}
Let $S$ be a $G$-equivariant vector bundle homomorphism of $E$, whose fibrewise norm is bounded. Then $S \in \calL_A(A)$, and $\|S\|_{\calL_A(A)} \leq \|S\|_{\infty}$.
\end{lemma}
\begin{proof}
The endomorphism $\|S\|_{\infty} 1_E - S$ of $E$ is fibrewise nonnegative. Let $T := (\|S\|_{\infty} 1_E - S)^{1/2}$ be its fibrewise square root. Then for all $\kappa \in C^*_{\ker}(M; L^2(E))^G_{\loc}$,
\[
\langle \kappa, (\|S\|_{\infty} - S) \kappa \rangle = \langle T\kappa, T \kappa \rangle \geq 0.
\]
So $S \leq \|S\|_{\infty}$ in $\calL_A(A)$.
\end{proof}

\begin{proof}[Proof of Theorem \ref{thm D reg}.]	
As pointed out at the start of the proof of Proposition~\ref{prop:cocompactregularity}, it is enough to prove
 that the operators $D\pm i$ have ranges that are dense in $C^*_{\max}(M;L^2(E))^G_{\loc}$. We prove this for $D+i$, with the case of $D-i$ being similar.
	
Analogously to the existence of the functions $f_{\varepsilon}$ in the proof of Proposition \ref{prop:cocompactregularity}, completeness of $M$ implies that	
 there exists a family $\{a_\varepsilon\}_{\varepsilon\in(0,1]}$ of $G$-invariant, cocompactly supported smooth functions taking values in $[0,1]$, such that:
	\begin{enumerate}
		\item $\textnormal{supp}(a_{\varepsilon_1})\subseteq \textnormal{supp}(a_{\varepsilon_2})$ whenever $\varepsilon_2\leq\varepsilon_1$;
		\item $a_{\varepsilon_1}^{-1}(1)\subseteq a_{\varepsilon_2}^{-1}(1)$ whenever $\varepsilon_2\leq\varepsilon_1$;
		\item $\bigcup_{\varepsilon}\{a_\varepsilon^{-1}(1)\}=M$;
		\item $\sup_{m\in M}\norm{da_\varepsilon(m)}\leq\varepsilon$.
	\end{enumerate}

	Now let $\kappa \in C_{\ker}^{\infty}(M;L^2(E))^G_{\loc}$. Since $\kappa$ has cocompact support, for small enough $\varepsilon$, we have $\textnormal{supp}(\kappa)\subseteq a_\varepsilon^{-1}(1)\times a_\varepsilon^{-1}(1)$. Let $U_\varepsilon$ be a $G$-invariant, relatively cocompact neighbourhood of $\supp(a_\varepsilon)$. Denote the double of its closure $\overline{U}_\varepsilon$ by $\overline{U}_\varepsilon^+$, noting that there exists a $G$-invariant collar neighbourhood of $\partial\overline{U}_{\varepsilon}$ inside $\overline{U}_\varepsilon$. By restricting the various geometric structures on $M$ to $U_\varepsilon$ and extending to $\overline{U}_\varepsilon^+$, we obtain a Dirac operator $D_\varepsilon$ acting on a bundle $E_\varepsilon$ over the double. 
	
	As in Subsection \ref{sec unbdd ops}, $D_\varepsilon$ defines an  operator on $C^{\infty}_{\ker}(\overline{U}_\varepsilon^+,L^2(E_\varepsilon))^G$ that extends to an unbounded operator on  the maximal completion $C^*_{\max}(\overline{U}_\varepsilon^+,L^2(E_\varepsilon))^G$, whose norm we will denote by $\norm{\cdot}_{\varepsilon,\max}$. By Proposition \ref{prop:cocompactregularity} and cocompactness of $\overline{U}_\varepsilon^+$, $D_\varepsilon$ is regular as an unbounded Hilbert $C^*$-module operator. So there exists a sequence
	$\{e_{\varepsilon,j}\}_{j\in\mathbb{N}}$ in $C^\infty_{\textnormal{ker}}(\overline{U}_\varepsilon^+,L^2(E_\varepsilon))^G$ such that
	\begin{equation}
		\label{eq:convergence}
		(D_\varepsilon+i)e_{\varepsilon,j}\rightarrow\kappa.
	\end{equation} in $\norm{\cdot}_{\varepsilon,\max}$. Since the action of the operator $(D_\varepsilon +i)^{-1}$ on $\kappa$ preserves support in the second coordinate, we may assume that
		$$\pr_2(\textnormal{supp}(e_{\varepsilon,j}))\subseteq \pr_2(\textnormal{supp}(\kappa)),$$
	where $\pr_2:\overline{U}_\varepsilon^+\times\overline{U}_\varepsilon^+\rightarrow\overline{U}_\varepsilon^+$ is the projection onto the second factor, so that $a_\varepsilon e_{\varepsilon,j}$ lies in the domain of $D$. For each $j$, 
	\beq{eq D+i ae}
	\begin{split}
	(D+i)(a_\varepsilon e_{\varepsilon,j})-\kappa&=(D_\varepsilon+i)(a_\varepsilon e_{\varepsilon,j})-\kappa\\
%	&=a_\varepsilon(D_\varepsilon+i)(e_{\varepsilon,j})+c(d a_\varepsilon)e_{\varepsilon,j}-\kappa\\
	&=a_\varepsilon((D_\varepsilon+i)e_{\varepsilon,j}-\kappa)+\sigma_D(d a_\varepsilon)e_{\varepsilon,j},
	\end{split}
	\eeq
 acting again on the first coordinate. 
 
  Because $D$ has finite propagation speed, there is a $C>0$ such that  $\| \sigma_D(\xi) \| \leq C\|\xi\|$ for all $\xi \in T^*M$. By Lemma \ref{lem vb end bdd}, this implies that $\|\sigma_D(d a_\varepsilon)\|_{\calL_A(A)} \leq C \|d a_{\varepsilon}\|_{\infty}$. 
% Because $D$ has finite propagation speed, there is a $C>0$ such that for all bounded one-forms $\alpha \in \Omega^1(M)$ and any element $\kappa'$ in the $*$-algebra $C_{\ker}^{\infty}(M;L^2(E))^G_{\loc}$,
%	$$(\sigma_D(\alpha)\kappa')^*(\sigma_D(\alpha)\kappa')-C\norm{\alpha}_\infty^2 \kappa'^*\kappa'$$
%	is negative, in the sense that it is equal to $-(S\kappa')^*S\kappa'$, where
%	$$S:=(C \norm{\alpha}_\infty^2-\sigma_D(\alpha)^*\sigma_D(\alpha))^{1/2}.$$
	It follows that 
	$$ \norm{\sigma_D(d a_\varepsilon)\kappa}_{\max}\leq C\norm{d a_\varepsilon}_\infty\norm{\kappa}_{\max}.$$
	Similarly,
	$$\norm{a_\varepsilon\kappa}_{\max}\leq\norm{\kappa}_{\max}.$$
	Thus $\sigma_D(d a_\varepsilon)$ and $a_\varepsilon$ define bounded multipliers of $C^*_{\max}(M;L^2(E))^G_{\loc}$ with norms bounded above by their supremum norms. Combining this with \eqref{eq D+i ae} gives
	\begin{align*}
	\norm{(D+i)(a_\varepsilon e_{\varepsilon,j})-\kappa}_{\max}&=\norm{a_\varepsilon((D_\varepsilon+i)e_{\varepsilon,j}-\kappa)+\sigma_D(d a_\varepsilon)e_{\varepsilon,j}}_{\max}\\
	&\leq\norm{(D_\varepsilon+i)e_{\varepsilon,j}-\kappa}_{\max}+C\norm{d a_\varepsilon}_\infty\norm{e_{\varepsilon,j}}_{\max}.
		\end{align*}
		
The algebra $C^*_{\max}\left(\overline{U}_\varepsilon; L^2(E|_{\overline{U}_\varepsilon})\right)^G$ is a subalgebra of the admissible Roe algebra $C^*_{\max}\left(\overline{U}_\varepsilon; L^2(E|_{\overline{U_{\varepsilon}}})\otimes L^2(G)\right)^G$, which is itself a common subalgebra of both $C^*_{\max}(M)^G_{\loc}$ and $C^*_{\max}(\overline{U}_\varepsilon^+; L^2(E_\varepsilon) \otimes L^2(G))^G$. 
%Here $E_{\overline{U}_\varepsilon^+} \to {\overline{U}_\varepsilon^+}$ is the extension of $E$.
 	This implies that for any kernel $\kappa' \in  C_{\ker}^{\infty}(M;L^2(E))^G_{\loc}$ supported on $U_\varepsilon\times U_\varepsilon$, 
\beq{eq kappa max epsilon}
\norm{\kappa'}_{\max} = \norm{\kappa'}_{\varepsilon,\max}
\eeq
	as both sides are equal to the norm of the image of $\kappa'$ in the algebra $C^*_{\max}\left(\overline{U}_\varepsilon; L^2(E|_{\overline{U_{\varepsilon}}})\otimes L^2(G)\right)^G$.
	Also note that \eqref{eq:convergence} implies that there exists $j_0$ (dependent on $\varepsilon$) such that for all $j\geq j_0$, 
\beq{eq e kappa}	
\norm{e_{\varepsilon, j}}_{\varepsilon,\max}\leq 2\norm{\kappa}_{\varepsilon,\max}.
		\eeq
By \eqref{eq kappa max epsilon} and \eqref{eq e kappa}, we have for all $j \geq j_0$,
%	Continuing the computation with these facts in mind gives
	\begin{align*}
	\norm{(D+i)(a_\varepsilon e_{\varepsilon,j})-\kappa}_{\max}
	&\leq\norm{(D_\varepsilon+i)e_{\varepsilon,j}-\kappa}_{\varepsilon,\max}+C\norm{d a_\varepsilon}_\infty\norm{e_{\varepsilon,j}}_{\varepsilon,\max}\\
	&\leq\norm{(D_\varepsilon+i)e_{\varepsilon,j}-\kappa}_{\varepsilon,\max}+2C\varepsilon
	%\norm{d a_\varepsilon}_\infty
	\norm{\kappa}_{\max}.
	\end{align*}
	
Now if any $\varepsilon'>0$ is given, choose $\varepsilon \in (0,1]$ so that $2C\varepsilon
	\norm{\kappa}_{\max} < \varepsilon'/2$.
For this $\varepsilon$, choose $j$ so that $j \geq j_0$, and 	$\norm{(D_\varepsilon+i)e_{\varepsilon,j}-\kappa}_{\varepsilon,\max} < \varepsilon'/2$. Then $\norm{(D+i)(a_\varepsilon e_{\varepsilon,j})-\kappa}_{\max} < \varepsilon'$.
So any element $\kappa$ of the dense subspace $C_{\ker}^{\infty}(M;L^2(E))^G_{\loc} \subset C_{\max}^{*}(M;L^2(E))^G_{\loc}$ can be approximated arbitrarily closely by elements in the image of $D+i$.	
	\end{proof}
	A straightforward adaptation of the above proof, together with Proposition \ref{prop:arbitrarysecondmanifoldcocpt}, gives the following result in the more general situation involving two different bundles and manifolds.
	\begin{theorem}
	\label{thm:arbitrarysecondmanifold}
	Let $M_1$ and $M_2$ be proper, isometric $G$-manifolds with and $E_1,E_2$ Hermitian $G$-vector bundles over $M_1$ and $M_2$ respectively. Suppose that $M_1$ is cocompact, and let $D$ be a symmetric, $G$-equivariant, first order differential operator acting on sections of $E_1$, with finite propagation speed.
	Then the unbounded operator 
	$$D:\HH_{\max}(E_1,E_2)^G_{\loc} \rightarrow\HH_{\max}(E_1,E_2)^G_{\loc} $$ 
	is regular and essentially self-adjoint.
%	As in Remark \ref{rem:arbitrarysecondmanifoldcocpt}, Theorem \ref{thm:regularity} the above proof still goes through if instead of $C_{\ker}^{\infty}(M;L^2(E))^G_{\loc}$ we took the algebra of smooth $G$-invariant cocompactly supported kernels on $M_1\times M_2$, where $M_1$ is a complete $G$-manifold and $M_2$ is any proper $G$-manifold (together with appropriate bundles $E_1\rightarrow M_1$ and $E_2\rightarrow M_2$).
	\end{theorem}

%%%%%%%%%%%%%

	\subsection{Generalised Fredholmness} \label{sec proof prop Roe max}

	We will prove Proposition \ref{prop Roe max} by adapting the method in \cite{Roe16} to the Hilbert $A$-module $A$. (Recall from Subsection \ref{sec loc index} that $A = C^*_{\max}(M;  L^2(E))^G_{\loc}$.)
	
	We begin by establishing a useful property of the wave operator group associated to an essentially self-adjoint regular operator, following Proposition 3.4 of \cite{Hanke-Pape-Schick}.
	\begin{lemma}
		\label{lem:waveoperators}
		Let $D$ be an essentially self-adjoint regular operator on $A$. Then the wave operator group $\left\{e^{itD}\right\}_{t\in\mathbb{R}}$ satisfies the wave equation: for $\kappa \in C_{\ker}^{\infty}(M;L^2(E))^G_{\loc}$, 
		$$\frac{d}{dt}e^{itD}\kappa=iD e^{itD}\kappa.$$
		Moreover, each operator $e^{itD}$ has propagation at most $|t|$, in the sense that it does not increase the propagation of $\kappa$ by more than $|t|$.
	\end{lemma}
	\begin{proof}
		The function $s\mapsto e^{its}$ is in $C_b(\mathbb{R})$. Thus for each $t\in\mathbb{R}$, the operator $e^{itD}$ is bounded adjointable and unitary. For each $t\in\mathbb{R}$, the difference quotient
		$$\frac{e^{i(t+h)s}-e^{its}}{h}\rightarrow ise^{its}$$
		as $h \to 0$,
		uniformly in $s$ in compact sets. Furthermore, this difference quotient is bounded uniformly by $|1+s|$. The wave equation property then follows from the third point in Theorem \ref{thm:functionalcalculus}. The finite propagation property can be proved exactly as in Proposition 3.4 in \cite{Hanke-Pape-Schick}.
	\end{proof}
%(\Todo: is this lemma true for any Hilbert module over any $C^*$-algebra? Then state and prove it in that generality.)	
	
%	\begin{corollary}
%		\label{cor:waveopsequal}
%		Let $K$ be a cocompact subset of $M$. Let $D_1$ and $D_2$ be essentially self-adjoint differential operators on $M$ that are equal on an open ball $B_r(K)$. Then for a kernel $a\in C_{\ker}^{\infty}(M;L^2(E))^G_{\loc}$ supported on $K\times K$, we have
%		$$e^{itD_1}a=e^{itD_2}a.$$
%	\end{corollary}
	\begin{corollary}
	\label{cor:waveopsequalgeneralmodule}
	Let $K$ be a cocompact subset of $M_1$. Let $r>0$. Let $D_1$ and $D_2$ be essentially self-adjoint differential operators on $M_1$ that are equal on  $\Pen(K, r)$. Then for a kernel $\kappa\in \HH_{\ker}^{\infty}(E_1,E_2)^G_{\loc} $ supported on $K\times M_2$,
	$$e^{itD_1}\kappa=e^{itD_2}\kappa$$
	if $-r\leq t\leq r$.
%	$$\cos(tD_1)a=\cos(tD_2)a,\qquad \sin(tD_1)a=\sin(tD_2)a.$$ 
	\end{corollary}
	
%(\Todo: use consistent notation $\Pen(X, r)$ or $B_r(X)$ everywhere.)
	
%	\begin{remark}
%		The analogue of Remark \ref{rem:cossin} also holds for $\HH_{\ker}^{\infty}(E_1,E_2)^G_{\loc} $.
%	\end{remark}

Now suppose that $D$ is as in Subsection \ref{sec loc index}, and in particular that it satisfies \eqref{eq D2 pos}.
We will use Corollary \ref{cor:waveopsequalgeneralmodule} to
establish a norm estimate for $f(D)$ in $\calL_A(A)$ when $f$ has compactly supported Fourier transform. 

\begin{lemma} \label{lem:positivity}
The operator $D$ on the Hilbert module $C^*_{\max}(M;L^2(E))^G_{\loc}$ satisfies 
	$$D^2\geq c^2$$
		with respect to the Hilbert module inner product. 
\end{lemma}
	\begin{proof}
For any $\kappa\in C_{\ker}^{\infty}(M;L^2(E))^G_{\loc}$,
\beq{eq D2 kappa}
	\langle D^2 \kappa,\kappa\rangle=\langle(P^*P+R)\kappa,\kappa\rangle
	=\langle P \kappa,P\kappa \rangle+\langle R \kappa,\kappa\rangle.
	\eeq
As in the proof of Lemma \ref{lem vb end bdd}, we find that $\langle R \kappa,\kappa\rangle \geq c^2 \langle \kappa, \kappa \rangle$. So the right hand side of \eqref{eq D2 kappa} is at least equal to
$c^2 \langle  \kappa,\kappa\rangle$.
	\end{proof}
Lemma \ref{lem:positivity} is the place where we use form \eqref{eq D2 pos} of $D^2$, rather than a slightly milder positivity condition on $D^2$ outside $Z$. We continue to use $c$ to denote the constant below \eqref{eq D2 pos}, which was also used in Lemma \ref{lem:positivity}.
	\begin{lemma}
		\label{lem:cptfourier}
		Let $r>0$.
		Suppose $f\in\mathcal{S}(\mathbb{R})$ is a function with Fourier transform $\hat{f}$ supported in $[-r,r]$. Let $\varphi$ be a smooth, bounded, $G$-invariant function with support disjoint from  $\Pen(Z, 2r)$. Then
		$$\norm{f(D)\varphi}_{\calL_A(A)}\leq \norm{\varphi}_\infty\cdot\sup\left\{|f(\lambda)|; |\lambda|\geq c\right\}.$$
		The same estimate applies to $\varphi f(D)$.
	\end{lemma}
	\begin{proof}
		For $n=1,2$, let $U_n=\{m\in M; d(m,Z)>nr\}.$ Let $\overline{U}_1^+$ denote the double of $\overline{U}_1$. By extending the various geometric structures on $U_1$ to $\overline{U}_1^+$, we may extend the Dirac operator $D|_{U_1}$ to an operator $\tilde{D}$ on $\overline{U}_1^+$ acting on the extension $\tilde{E}\rightarrow\overline{U}_1^+$ of $E|_{\overline{U_1}}$. 
		%Let $\tilde{Z}$ be any cocompact subset of $\overline{U}_1^+$. 
		%(\Todo: so $Z$ is not assumed to be cocompact here? Where do we assume this, and where not?)
%		Write 
%		\begin{align*}
%		\HH_{\ker}^{\infty}&:=\HH_{\ker}^{\infty}(\overline{U}_1^+,M)^G_{\loc},\\
%		\HH_{\max}&:=\HH_{\max}(\overline{U}_1^+,M)^G_{\loc}.
%		\end{align*}
		Then $\tilde{D}$ is an unbounded symmetric operator on $\HH_{\max}(\overline{U}_1^+,M)^G_{\loc}$ with initial domain $\HH_{\ker}^{\infty}(\overline{U}_1^+,M)^G_{\loc}$. By Lemma \ref{lem:positivity}, we have  $\tilde{D}^2\geq c^2$. Since $\overline{U}_1^+$ is complete, Theorem \ref{thm:arbitrarysecondmanifold} implies that 
		$\tilde{D}$ is essentially self-adjoint and regular.

		Now for any $\kappa\in C_{\ker}^{\infty}(M;L^2(E))^G_{\loc}$ with support contained in $U_2\times M$, Corollary \ref{cor:waveopsequalgeneralmodule} implies that
		$$e^{itD}\kappa=e^{it\tilde{D}}\kappa$$
		for all $-r\leq t\leq r$. 
Together with the equality
			$$f(D)=\frac{1}{\pi}\int_{-r}^r\hat{f}(t)e^{itD}\,dt,$$		
this implies that
$f(D)\varphi=f(\tilde{D})\varphi$. The bound $\tilde{D}^2\geq c^2$ implies, by the fifth point of Theorem \ref{thm:functionalcalculus}, that		
$$\|{f(\tilde{D})}\|_{\calL_A(A)}\leq\sup\{|f(\lambda)|;|\lambda|\geq c\}.$$
		Together with the fact that $\varphi$ defines an element of ${\calL_A(A)}$ with norm at most $\norm{\varphi}_\infty$ (a special case of Lemma \ref{lem vb end bdd}), this gives
		\begin{align*}
		\|f(D)\varphi\|_{\calL_A(A)}\leq\norm{\varphi}_\infty\cdot\sup\{|f(\lambda)|;|\lambda|\geq c\}.&\qedhere
		\end{align*}
	\end{proof}
	With these preparations, we can now finish the proof of Proposition \ref{prop Roe max}. Again, we follow the idea of Roe \cite{Roe16}.
	\begin{proof}[Proof of Proposition \ref{prop Roe max}.]
		Let $f\in C_c(\R)$ be supported in $[-c,c]$, and let $\varepsilon>0$. There exists a smooth function $g$ with compactly supported Fourier transform such that
		$$\sup\{|g(\lambda)-f(\lambda)|;\lambda\in\mathbb{R}\}<\varepsilon.$$
		This implies that $|g(\lambda)|<\varepsilon$ when $|\lambda|>c$. Suppose that $\textnormal{supp}(\hat{g})\subseteq [-r,r]$ for some $r>0$. Let $\psi:M\rightarrow[0,1]$ be a smooth $G$-invariant function such that
		$$\psi(m)=\begin{cases}1&\textnormal{ for }m\in \Pen(Z, 2r)\\
		0&\textnormal{ for }m\in M\backslash \Pen(Z, 3r).\end{cases}$$
		We can write
		$$f(D)=\psi g(D)\psi+(1-\psi)g(D)\psi+g(D)(1-\psi)+(f(D)-g(D)).$$
		Now the first term on the right hand side is a $G$-invariant cocompactly supported smooth kernel. 
		%To see that it is smooth, note that $g(D)=g(D')$, where $D'$ is the essentially self-adjoint Dirac operator on the double of $\overline{B_{6r}(Z)}$, a cocompact $G$-manifold without boundary, whence $g(D')$ has smooth kernel.
		The second and third terms each have maximal norm bounded by $\varepsilon$ by Lemma \ref{lem:cptfourier}, while the maximal norm of the last term is bounded by $\varepsilon$ by the fifth point in Theorem \ref{thm:functionalcalculus}. Thus, for any $\varepsilon>0$, $f(D)$ lies within $3\varepsilon$ of a $G$-invariant cocompactly supported smooth kernel. Thus $f(D)$ is in the completion $C^*_{\max}(M;L^2(E))^G_{\loc}$.
	\end{proof}

\section{Averaging maps} \label{sec avg}

In Subsections \ref{sec Av ker}--\ref{sec Av I}, we return to the general setting of Subsection \ref{sec C0X mod}, of a metric space $(X,d)$ rather than a Riemannian manifold.

The main tools in the proof of Theorem \ref{thm invar index} are several \emph{averaging maps}, which map $G$-equivariant operators on $X$ to operators on $X/G$. In this section, we introduce these maps, and prove their properties that we need. We then use these maps in Section \ref{sec pf invar index} to prove Theorem \ref{thm invar index}.

\subsection{Averaging kernels} \label{sec Av ker}

% Define Av, prove it is surjective

Consider the setting of Subsection \ref{sec C0X mod}.
 Consider the action by $G \times G$ on $\Gamma(\Hom(E))$ given by
\beq{eq action GG}
((g,g')\cdot \kappa)(x,x') := g\kappa(g^{-1}x,g'x')g',
\eeq
for $g,g' \in G$, $x, x' \in X$ and $\kappa \in \Gamma(\Hom(E))$. Let $\Gamma(\Hom(E))^{G \times G} \subset \Gamma(\Hom(E))$ be the subspace of sections invariant under this action. 

Let $d_G$ be the metric on $X/G$ induced by $d$:
\[
d_G(Gx, Gx') := \inf_{g \in G}d(gx,x'),
\]
for $x,x' \in X$.
Consider
the measure $d(Gx)$ on $X/G$ such that for all $\varphi \in C_c(X)$,
 \beq{eq meas dGx}
 \int_X \varphi(x)\, dx = \int_{X/G} \int_G \varphi(gx)\, dg \, d(Gx).
 \eeq
(See for example \cite{BourbakiVI}, Chapter VII, Section 2.2, Proposition 4b.)

%A section $\kappa \in \Gamma(\Hom(E))^{G\times G}$ has  \emph{transversally finite propagation} if there is an $r>0$ such that for all $x, x' \in X$ with $d_G(Gx, Gx')>r$, one has $\kappa(x, x')=0$. 
%If $\kappa$ has transversally finite propagation, the infimum of such $r$ is called the \emph{transversal propagation} of $a$.

Consider the Hilbert space $L^2_T(E)^G$ defined in Subsection \ref{sec invar index}. We view it as a $C_0(X/G)$-module by pointwise multiplication after pullback along the quotient map. Let $C^*_{\ker}(X/G; L^2_T(E)^G)$ be the subalgebra of $\cB(L^2_T(E)^G)$ of locally compact operators $T$ with finite propagation,  given by a continuous kernel $\kappa \in  \Gamma(\Hom(E))^{G\times G}$ via
\[
(Ts)(x) = \int_{X/G} \kappa(x,x')s(x')\, d(Gx'),
\]
for $x \in X$ and $s \in L^2_T(E)^G$. The integral is independent of the Borel section $X/G \to X$ used implicitly, by $G$-invariance of $s$ and $G \times G$-invariance of $\kappa$. We will identify operators in $C^*_{\ker}(X/G; L^2_T(E)^G)$ with their kernels.

For $\kappa \in C^*_{\ker}(X; L^2(E))^G$ and $x,x' \in X$, set
\[
\Av(\kappa)(x,x') := \int_{G} g \kappa(g^{-1}x,x')\, dg.
\]
The integrand is bounded, measurable and compactly supported, because $\kappa$ has finite propagation and the action is proper.
\begin{lemma} \label{lem Av ker hom}
For every $\kappa \in C^*_{\ker}(X; L^2(E))^G$, $\Av(\kappa)$ is an element of $C^*_{\ker}(X/G; L^2_T(E)^G)$. This defines a surjective $*$-homomorphism
\[
\Av\colon C^*_{\ker}(X; L^2(E))^G \to C^*_{\ker}(X/G; L^2_T(E)^G).
\]
\end{lemma}
\begin{proof}
Let  $\kappa \in C^*_{\ker}(X; L^2(E))^G$. 
It follows from a computation involving $G$-equivariance of $\kappa$ and  left and right invariance of $dg$ that
$\Av(\kappa)$ is $G \times G$-invariant.
 Furthermore, the propagation of $\Av(\kappa)$ in $X/G$ is at most equal to the propagation of $\kappa$ in $X$.  
So indeed $\Av(\kappa) \in C^*_{\ker}(X/G; L^2_T(E)^G)$.

Using the fact that  the Haar measure $dg$ is invariant under inversion, one computes directly that
for all $\kappa \in C^*_{\ker}(X; L^2(E))^G$,
\[
\Av(\kappa^*) = \Av(\kappa)^*.
\]
Let $\kappa, \kappa' \in C^*_{\ker}(X; L^2(E))^G$ be given. Then a computation involving \eqref{eq meas dGx} shows that
\[
\Av(\kappa \kappa') = \Av(\kappa)\Av(\kappa').
\]

For surjectivity, let $\chi$ be a cutoff function as in \eqref{eq chi cutoff}. Choose this function such that, in addition to its other properties, 
\beq{eq supp chi}
d(x,x') \leq  d_G(Gx, Gx') +1
\eeq
 for all $x,x' \in \supp(\chi)$. In other words, the support of $\chi$ mainly extends transversally to $G$-orbits. (See Lemma \ref{lem cutoff transv} below.)

Suppose that $\kappa_G \in C^*_{\ker}(X/G; L^2_T(E)^G)$, and let $r$ be its  propagation in $X/G$. Define $\kappa \in \Gamma(\Hom(E))$ by
\[
\kappa(x,x') = \int_G \chi(gx)^2 \chi(gx')^2\, dg \cdot \kappa_G(x,x'),
\]
for $x,x' \in X$.

We first claim that $\kappa$ has finite propagation. Indeed, let $x,x' \in X$ be such that $d(x,x')>r+1$. If $d_G(Gx,Gx') >r$, then $\kappa_G(x,x')=0$, so $\kappa(x,x')=0$.
So suppose that $d_G(Gx,Gx') \leq r$. If $g \in G$, and $\chi(gx)^2 \chi(gx')^2$ is nonzero, then \eqref{eq supp chi} implies that
\[
d(x,x') = d(gx,gx') \leq r+1,
\]
a contradiction. So $\chi(gx)^2 \chi(gx')^2=0$ for all $g \in G$, and hence $\kappa(x,x')=0$.

Right invariance of $dg$ implies that for all $x,x' \in X$ and $g \in G$,
\[
g^{-1}\kappa(gx, gx')g = \kappa(x,x'),
\]
i.e.\ $\kappa$ is invariant under the restriction of the action \eqref{eq action GG} to the diagonal. This implies that the operator on $L^2(E)$ defined by $\kappa$ is $G$-equivariant.

The property \eqref{eq chi cutoff} of $\chi$ and left invariance of $dg$ imply that
\[
\Av(\kappa) = \kappa_G,
\]
so $\Av$ is surjective.
\end{proof}

\begin{lemma} \label{lem cutoff transv}
There is a cutoff function $\chi$, satisfying \eqref{eq chi cutoff}, such that  for all $x,x' \in \supp(\chi)$,
\[
d(x,x') \leq d_G(Gx, Gx')+1.
\]
\end{lemma}
\begin{proof}
Let $Y \subset X$ be a subset intersecting all $G$-orbits, such that for all $y,y' \in Y$,
\[
d(y,y') \leq  d_G(Gy, Gy')+1/2.
\]
%(\Todo: comment why this is possible?)
Let $U$ be the open set of all points in $X$ closer than $1/8$ to $Y$. Then the intersection of $U$ with any $G$-orbit is open in that orbit. For all $u,u' \in U$, choose $y,y' \in Y$ such that $d(u,y) < 1/8$ and $d(u',y') < 1/8$. Then
\[
d(u,u') < d(y,y')+1/4 \leq d_G(Gy,Gy')+3/4 < d_G(Gu,Gu')+1.
\]
Let $\tilde \chi$ be any nonnegative continuous function on $X$ such that the interior of its support is $U$. Then the function $\chi$, given by
\[
\chi(x) := \frac{\tilde \chi(x)}{\Bigl(\int_G \tilde \chi(gx)^2 \, dg\Bigr)^{1/2}}
\]
for $x \in X$,
has the desired properties.
\end{proof}

\subsection{Averaging operators on Hilbert spaces} \label{sec Av op}

We will use an extension of the homomorphism $\Av$ to more general bounded operators on $L^2(E)$, not necessarily given by integrable kernels. 

We choose a partition of unity $\{\psi_j\}_{j=1}^{\infty}$ on $X$, which restricts to a compactly supported partition of unity on every orbit, such that
%\begin{itemize}
%\item there is a Borel section $\sigma: X/G \to X$ whose image intersects the support of only finitely many of the functions $\psi_j$;
%\item 
there is an $r>0$ such that for all $j$, the set of $k$ for which 
$d(\supp \psi_j, \supp \psi_k) \leq r$
is finite. 
%\end{itemize}

Let $\cB_{\fp}(L^2(E))^G$ be the algebra of $G$-equivariant, bounded operators on $L^2(E)$ with finite propagation.
Given $T \in \cB_{\fp}(L^2(E))^G$  and $s \in L^2_T(E)^G$, note that $\psi_j s \in L^2(E)$. Hence $T(\psi_j s)$ is well-defined, and for $x \in X$, we set
\begin{equation} \label{eq av T}
(\Av_{L^2}(T)s)(x) := \sum_j (T(\psi_j s))(x).
\end{equation}
\begin{lemma} \label{lem Av op well def}
The sum \eqref{eq av T} converges for all $x \in X$, and the result is independent of the choice of the partition of unity $\{\psi_j\}$.
\end{lemma}
\begin{proof}
Fix $x \in X$. Let $f \in C_c(X)$ be a function such that $f(x) = 1$. Let $r$ be greater than both the diameter of $\supp f$ and the propagation of $T$, and such that
\[
d(\supp \psi_j, \supp \psi_k) \leq 2r
\]
for only finitely many $k$, for any fixed $j$. Then there are only finitely many $j$ for which
\[
d(\supp f, \supp \psi_j) \leq r.
\]
Hence $fT\psi_j$ is nonzero for only finitely many $j$, and we see that
\[
(\Av_{L^2}(T)s)(x) = (f\Av(T)s)(x) = \sum_j (fT(\psi_j s))(x)
%
%
%\bigl(f\Avpi(T)\varphi\bigr)(x) = \sum_j \bigl(f(T\otimes 1_{\Hpi}) (\psi_j \varphi)\bigr)(x)
\]
is a finite sum, and hence converges.

Let $\{\psi_j'\}$ be another partition of unity on $X$, with the same properties as $\{\psi_j\}$. We will write
\[
\Av_{L^2}^{\{\psi_j\}}(T) \quad \text{and} \quad \Av_{L^2}^{\{\psi'_j\}}(T) 
\]
for the operators defined by \eqref{eq av T} using these two partitions of unity. As above, let $J \subset \N$ be a finite set such that $fT\psi_j$ and $fT\psi'_j$ are zero if $j \not\in J$. Then
\[
fT\sum_{j \in J} \psi_j = fT = fT\sum_{j \in J} \psi'_j. 
\]
Since the finite sum over $J$ commutes with $T$, we conclude that

\begin{multline*}
(\Av_{L^2}^{\{\psi_j\}}(T) s)(x) =(f\Av_{L^2}^{\{\psi_j\}}(T) s)(x) 
= \sum_{j\in J} (fT(\psi_j s))(x)\\
= \sum_{j\in J} (fT(\psi_j' s))(x)
= (\Av_{L^2}^{\{\psi_j'\}}(T) s)(x).
\end{multline*}
%%
%%
% \bigl(f\Avpi_{\{\psi_j\}}(T) \varphi\bigr)(x) \\
%	&= \sum_{j \in J} \bigl(f(T\otimes 1_{\Hpi}) (\psi_j \varphi)\bigr)(x) \\
%	& = \left( \biggl( \Bigl( fT \sum_{j \in J}  \psi_j    \Bigr)  \otimes 1_{\Hpi} \biggr)
%\varphi\right)(x) \\
%	& = \left( \biggl( \Bigl( fT \sum_{j \in J}  \psi'_j    \Bigr)  \otimes 1_{\Hpi} \biggr)
%\varphi\right)(x) \\
%	&=\bigl(\Avpi_{\{\psi'_j\}}(T) \varphi\bigr)(x) 

\end{proof}

%(\Todo: $T$ has a few different meanings here....?)

\begin{lemma}
The construction \eqref{eq av T} defines a $*$-homomorphism
\[
\Av_{L^2}\colon \cB_{\fp}(L^2(E))^G \to \cB(L^2_T(E)^G).
\]
\end{lemma}
\begin{proof}
Let $T \in  \cB_{\fp}(L^2(E))^G $ and $s \in L^2_T(E)^G$. We first claim that $\Av_{L^2}(T)s \in L^2_T(E)^G$. Indeed, the properties of the functions $\psi_j$, and finite propagation of $T$ imply that $\chi \Av(T)s \in L^2(E)$. And for all $g \in G$ and $x \in X$, one checks, using $G$-equivariance of $T$ and $G$-invariance of $s$, that
\[
(g\cdot(\Av(T)s))(x) = \sum_{j} T( (g\cdot\psi_j) s)(x).
\]
Since $\{g\cdot\psi_j\}_{j=1}^{\infty}$ is a partition of unity with the same properties as  $\{\psi_j\}_{j=1}^{\infty}$, the second part of Lemma \ref{lem Av op well def} implies that the right hand side equals $(\Av_{L^2}(T)s)(x)$.

Boundedness of the operator $\Av_{L^2}(T)$ on $L^2_T(E)^G$ follows from boundedness and finite propagation of $T$, via the fact that the sum $\chi \sum_j T\psi_j$ is finite.

If $T'$ is an other operator in $\cB_{\fp}(L^2(E))^G$, then
 for all $x \in X$,
\begin{multline*}
(\Av_{L^2}(TT')s)(x) = \sum_j(TT'\psi_j s)(x)\\= \sum_{j,k}(T\psi_kT'\psi_j s)(x)
= (\Av_{L^2}(T)\Av_{L^2}(T')s)(x).
\end{multline*}
Here we used the fact that the sum over $k$ is finite for each $j$, since $T'$ has finite propagation. One checks directly that $\Av_{L^2}$ preserves $*$-operations.
\end{proof}

\begin{lemma} \label{lem Av op ker}
For an operator $T \in C^*_{\ker}(X, L^2(E))^G$, with kernel $\kappa$, the operator $\Av_{L^2}(T)$ is an element of $C^*_{\ker}(X/G; L^2_T(E)^G)$, and its kernel is $\Av(\kappa)$.
\end{lemma}
\begin{proof}
This follows from a direct computation involving \eqref{eq meas dGx}.
\end{proof}
If $s_1$ and $s_2$ are sections of $E$ such that their pointwise inner product $(s_1, s_2)_{E}$ is in $L^1(X)$, we will say that the inner product $(s_1, s_2)_{L^2(E)}$ converges, and define it as the integral of $(s_1, s_2)_{E}$ over $X$.
\begin{lemma} \label{lem fin prop est norm}
For all $T \in \cB_{\fp}(L^2(E))^G$,  $s \in \Gamma_{\tc}(E)^G$ and $\sigma \in \Gamma_c(E)$, 
%\begin{align}
%|(s, T\sigma)_{L^2(E)}| &\leq \|T\| |(s, \sigma)_{L^2(E)}|; \label{eq est sTsig}\\
\[
 |(\Av_{L^2}(T)s, \sigma)_{L^2(E)}| \leq \|T\| |(s, \sigma)_{L^2(E)}|.\label{eq est AvTssig}
\]
% \end{align}
In particular, this inner product converges.
\end{lemma}
\begin{proof}
%For \eqref{eq est sTsig}, let $f \in C_c(X)$ be equal to $1$ within the propagation of $T$ from the support of $\sigma$. Then
%\[
%|(s, T\sigma)_{L^2(E)}| = |(fs, T\sigma)_{L^2(E)}| \leq \|T\|  |(fs, \sigma)_{L^2(E)}| = \|T\|  |(s, \sigma)_{L^2(E)}|.
%\]
%
%For \eqref{eq est AvTssig}, 
Let $J \subset \N$ be a finite subset (depending on $T$, $s$ and $\sigma$) such that $\sum_{j \in J}\psi_j = 1$ on $\supp(\sigma)$, and  for all $j \in \N\setminus J$, $(T\psi_j s, \sigma)_E = 0$. Then
\begin{multline*}
\left|(\Av(T)s, \sigma)_{L^2(E)}\right| = |\bigl(T \sum_{j\in J} \psi_j s, \sigma \bigr)_{L^2(E)}| \\
\leq \|T\|  |\bigl( \sum_{j\in J} \psi_j s, \sigma \bigr)_{L^2(E)}|
= \|T\|  |( s, \sigma )_{L^2(E)}|.\qedhere
\end{multline*}
\end{proof}

\subsection{Relation to the integration trace} \label{sec Av I}

In this subsection, we suppose that $X/G$ is compact. Consider the map $\oplus 0$ in \eqref{eq def plus 0}.
\begin{lemma} \label{lem Av I norm}
For all $\kappa \in C^*_{\ker}(X; L^2(E))^G$,
\[
\|\Av({\kappa})\|_{\cB(L^2_{T}(E)^G)} = \| (I\otimes 1)({\kappa} \oplus 0)\|_{\cB(\HH)}.
\]
\end{lemma}

\begin{lemma} \label{lem Av max cocpt}
If $X/G$ is compact, then $\Av$ has a unique extension to a $*$-homomorphism
\[
\Av\colon C^*_{\max}(X; L^2(E))^G \to \cK(L^2_T(E)^G).
\]
\end{lemma}
\begin{proof}
If $X/G$ is compact, then $C^*_{\ker}(X/G; L^2_T(E)^G) \subset \cK(L^2_T(E)^G)$. So  the claim is that $\Av$ is continuous with respect to the norm $\|\cdot\|_{\max}$ and the operator norm on  $\cK(L^2_T(E)^G)$. And Lemma \ref{lem Av I norm} implies that for all $\kappa \in C^*_{\ker}(M; L^2(E))^G$,
\[
\|\Av({\kappa})\|_{\cB(L^2_{T}(E))^G}  \leq \|\kappa \oplus 0\|_{\max} = \|\kappa\|_{\max},
\]
because $I \otimes 1$ is a $*$-representation of $C^*_{\ker}(X; L^2(E) \otimes L^2(G))^G \subset L^1(G) \otimes \cK(\HH)$.
\end{proof}

\begin{proposition} \label{prop Avpi pi}
\begin{enumerate}
The following diagram commutes:
\[
\xymatrix{
K_0(C^*_{\max}(X; L^2(E))^G) \ar[r]^-{\oplus\,0} \ar[d]_-{\Av} & 
	K_0(C^*_{\max}(X)^G) \ar[d]_-{\cong} \\
K_0(\cK) = \Z& \ar[l]_-{I_*}K_0(C^*_{\max}(G)). 
 }
\]
\end{enumerate}
\end{proposition}

As in Section 4.4 of \cite{HWW}, consider the map
\[
\tilTR\colon C^*_{\ker}(X; L^2(E))^G \to L^1(G) \otimes \cK(L^2(E))
\]
such that for $\kappa \in  C^*_{\ker}(X; L^2(E))^G$ and $g \in G$, $\tilTR(\kappa)(g)$ is the operator with compactly supported Schwartz kernel given by
\[
\bigl(\tilTR(\kappa)(g)\bigr)(x,x') = \chi(x)\chi(x')g\kappa(g^{-1}x,x'),
\]
for $x,x' \in X$.
We will use the following fact in the proofs of Lemma \ref{lem Av I norm} and Proposition \ref{prop Avpi pi}.
\begin{proposition}\label{prop Avpi TR}
There is a unitary isomorphism $\eta\colon\HH \xrightarrow{\cong}L^2(E)$ such that for every conjugation-invariant map $\tau\colon L^1(G) \to \C$, the following diagram commutes:
\[
\xymatrix{
C^*_{\ker}(X; L^2(E))^G \ar[r]^-{\oplus\,0} \ar[d]_-{\tilTR} & C^*_{\ker}(X;  L^2(E) \otimes L^2(G))^G \ar[d]_-{\cong} \\
L^1(G) \otimes \cK(L^2(E)) \ar[d]_-{\tau \otimes 1}& L^1(G) \otimes \cK(\HH) \ar[d]_-{\tau \otimes 1}\\
\cK(L^2(E)) & \cK(\HH) \ar[l]_-{\eta_*}
 }
\]
\end{proposition}
This is essentially Proposition 4.10 in \cite{HWW}. There a specific map $\tau$ is used, but its only property used in the proof of this proposition is conjugation invariance.

%For $\kappa \in C^*_{\ker}(X; L^2(E))^G$, let  $T_{\kappa} \in C^*_{\alg}(X; L^2(E))^G$ be the operator with Schwartz kernel $\kappa$.
%\Todo: check that $T_{\kappa} \in C^*_{\alg}(M; L^2(E))^G$, probably did this in another file.

Consider the isometric embedding $j_{\chi}\colon L^2_{T}(E)^G \to L^2(E)$ given by pointwise multiplication by $\chi$. It induces $(j_{\chi})_*\colon \cK(L^2_{T}(E)^G) \to \cK(L^2(E))$, given by
\[
(j_{\chi})_*(T)s = \left\{
\begin{array}{ll}
\chi T \sigma & \text{if $s = \chi\sigma$ for $\sigma \in L^2_T(E)^G$;}\\
0 & \text{if $s \in (\chi L^2_T(E)^G)^{\perp}$,}
\end{array}
\right.
\]
for all $T \in \cK(L^2_T(E)^G)$ and $s \in L^2(E)$.
Consider the map
\[
I \otimes 1\colon L^1(G) \otimes \cK(L^2(E)) \to \cK(L^2(E)).
\]
\begin{lemma} \label{lem Avpi 2}
For all $\kappa \in C^*_{\ker}(M; L^2(E))^G$, 
\[
 (I\otimes 1) \circ \tilTR (\kappa) = (j_{\chi})_* ( \Av(T_{\kappa})). 
 \]
\end{lemma}
\begin{proof}
Let  $\kappa \in C^*_{\ker}(X; L^2(E))^G$, and write
\[
T :=  (I\otimes 1)( \tilTR (\kappa)). 
\]
Write
\[
T = \left( 
\begin{array}{cc}
a & b \\ c & d
\end{array}
\right)
\]
with respect to the decomposition $L^2(E) = \chi L^2_T(E)^G \oplus (\chi L^2_T(E)^G)^{\perp}$.

If $s \in L^2(E)$ and $x \in X$, then it follows from the definitions that
\beq{eq Ts}
(Ts)(x) = \chi(x) \int_X \int_G \chi(x') g\kappa(g^{-1}x, x')s(x')\, dg\, dx'.
\eeq
For $x \in X$, set
\[
\sigma(x) :=  \int_X \int_G \chi(x') g\kappa(g^{-1}x, x')s(x')\, dg\, dx'.
\]
It follows from left invariance of $dg$ that $\sigma$ is a $G$-invariant section of $E$. And $\chi\sigma \in L^2(E)$, so $\sigma \in L^2_T(E)^G$. So the image of $T$ lies inside $\chi L^2_T(E)^G$, which means that $c=d=0$.

Because $I$ and $\tilTR$ are $*$-homomorphisms, 
\[
 \left( 
\begin{array}{cc}
a^* & 0 \\ b^* & 0
\end{array}
\right) = T^* = (I\otimes 1)( \tilTR (\kappa^*)).
\]
Since the image of $(I\otimes 1)( \tilTR (\kappa^*))$ lies inside $\chi L^2_T(E)^G$ by the same argument as for $\kappa$, we find that $b^* = 0$, so $b = 0$.

If  $\sigma \in L^2_T(E)^G$, then it follows from \eqref{eq Ts}, $G$-invariance of $\kappa$, $\sigma$ and $dm$, and from \eqref{eq chi cutoff} that
\[
T\chi \sigma = \chi \Av(\kappa)\sigma.
\] 
So $a = j_{\chi}  \circ \Av({\kappa}) \circ j_{\chi}^{-1}$.
%
%$s \in L^2_{\pi}(E)$ and $m \in M$, then by definition
%\beq{eq Avpi TR}
%\bigl((I \otimes 1) \circ \tilTR (\kappa)\chi s\bigr) = \chi(m)\int_G \int_{M}
%\chi(m')^2 \bigl(g\kappa(g^{-1}m,m') \bigr) s(m')\, dm'\, dg.
%\eeq
%The $G$-invariance properties of $\kappa$ and $s$ imply that for all $m' \in M$ and $g \in G$,
%\[
%g\kappa(g^{-1}m,m') s(m') =  \kappa(m,gm') s(gm').
%\]
%Substituting $m'' = gm'$ and using \eqref{eq chi cutoff}, we find that the right hand side of \eqref{eq Avpi TR} equals
%\beq{eq chi kappa}
%\chi(m) \int_{M}\kappa(m,m')s(m')\, dm'.
%\eeq
%Because $\kappa$ has finite propagation, there is a number $n \in \N$ such that
%\[
%\chi T_{\kappa} = \sum_{j=1}^n \chi T_{\kappa} \psi_j.
%\]
%So \eqref{eq chi kappa} equals $(\chi \Av(T_{\kappa})s)(m)$.
\end{proof}

\begin{proof}[Proof of Lemma \ref{lem Av I norm} and Proposition \ref{prop Avpi pi}.]
Proposition \ref{prop Avpi TR}, with $\tau = I$, and
 Lemma \ref{lem Avpi 2} imply
that the diagram
\[
\xymatrix{
C^*_{\ker}(X; L^2(E))^G \ar[r]^-{\oplus\,0} \ar[d]_-{\Av} & C^*_{\ker}(X;  L^2(E) \otimes L^2(G))^G \ar[d]\\
\cK(L^2_T(E)^G) \ar[d]_-{(j_{\chi})_*}& L^1(G) \otimes \cK(\HH) \ar[d]_-{I \otimes 1}\\
\cK(L^2(E)) & \cK(\HH) \ar[l]_-{\eta_*}
 }
\]
commutes. This implies both Lemma \ref{lem Av I norm} and Proposition \ref{prop Avpi pi}.
\end{proof}

\subsection{Averaging operators on Hilbert $C^*$-modules} \label{sec Av Hilb Cstar}

% Define extension to multiplier alg
% Lemma 4 in Compatibility note

As in Subsection \ref{sec loc index}, we write $A:= C^*_{\max}(X; L^2(E))^G_{\loc}$. 
By \eqref{eq Cker inj lim} and Lemmas \ref{lem Av ker hom} and \ref{lem Av max cocpt}, we obtain a surjective $*$-homomorphism
\[
\Av\colon A \to \cK(L^2_T(E)^G).
\]
This extends uniquely to multiplier algebras, giving
\beq{eq til Av}
\tilAv\colon \calL_A(A) \to \cB(L^2_T(E)^G).
\eeq

We return to the setting where $X = M$ is a Riemannian manifold, as in Subsection \ref{sec loc index}.

For clarity, we will use subscripts $A$ and $L^2(E)$ to denote functional calculus of operators on the Hilbert module $A$ and on the Hilbert space $L^2(E)$, as in the following lemma.
\begin{lemma} \label{lem fc Hilb cstar}
Let $f \in C_b(\R)$ such that $f \in C_0(\R)$, or $f(x) = \cO(x)$ as $x \to 0$. 
By functional calculus on the Hilbert $A$-module $A$, we can form the operator
\[
f(D)_{A} \in \calL_{A}(A).
\]
Via the usual functional calculus, we can form the operator
\[
f(D)_{L^2(E)} \in \cB(L^2(E))^G.
\]
Let $\kappa \in C_{\ker}^{\infty}(M;L^2(E))^G_{\loc}$. 
Let $T_{\kappa} \in \cB(L^2(E))^G$ be the operator with Schwartz kernel ${\kappa}$.
Then the operator $f(D)_{L^2(E)} \circ T_{\kappa}$ has a smooth kernel, and so does 
%
%, and let $f(D)_{L^2(E)} a$ be the smooth kernel of $f(D)_{L^2(E)} \circ T_a$.
%Then 
\[
f(D)_{A}({\kappa}). % \in C_{\ker}^{\infty}(M;L^2(E))^G_{\loc}.
\]
These two smooth kernels are equal.
%and its smooth kernel is $f(D)_{L^2(E)} a$.
\end{lemma}
\begin{proof}
%If $f \in C_b(\R)$, let $f(D)_{L^2(E)} \in \cB(L^2(E))^G$ be  defined by functional calculus of operators on $L^2(E)$.
 For an operator $S \in \cB(L^2(E))^G$ with a smooth kernel with finite propagation, 
 %the composition $A \circ f(D)$maps $L^2(E)$ continously into the space of smooth sections, and therefore has a smooth kernel. 
 and for all $n \in \N$, the operator $D^n f(D)S = f(D) D^nS$ is a bounded operator on $L^2(E)$. (Indeed, $D^n S$ has a smooth, $G$-invariant kernel with finite propagation, so it defines a  bounded operator since $M/G$ is compact.) Since $D$ is elliptic, it follows that the image of $f(D)S$ lies in the smooth sections, so that this operator also has a smooth kernel. Let $f(D)_{L^2(E)} {\kappa}$ be the smooth kernel of $f(D)_{L^2(E)} \circ T_{\kappa}$.

Next, suppose that $f(x) = (x \pm i)^{-1}$. The unbounded operator $D \pm i$ on $A$ is given by applying $D\pm i$ to the first coordinate of a smooth kernel. The element $f(D)_{A}({\kappa})$ lies in the domain of this operator, and
\[
(D \pm i) (f(D)_{A}({\kappa})) = {\kappa} = (D\pm i) (f(D)_{L^2(E)} {\kappa}).
\]
So the Schwartz kernels of $f(D)_{A}({\kappa})$ and  $f(D)_{L^2(E)} {\kappa}$ are equal.
  Since the functions $x\mapsto (x\pm i)^{-1}$ generate $C_0(\R)$, the claim follows for all $f \in C_0(\R)$.

Now suppose that $f \in C_b(\R)$, and $f(x) = \cO(x)$ as $x \to 0$. Then $g(x) = f(x)/x$ defines a function  $g \in C_0(\R)$. The preceding arguments imply that 
\begin{align*}
f(D)_{A}({\kappa}) =
D g(D)_{A}({\kappa}) 
=
 D g(D)_{L^2(E)} {\kappa}  =
  f(D)_{L^2(E)} {\kappa}. &\qedhere 
\end{align*}
\end{proof}

We will use the following relation between the averaging maps in \eqref{eq av T} and \eqref{eq til Av}   in the proof of Theorem \ref{thm invar index}.
	\begin{lemma}
		\label{lem:Avsequal}		
Let $b \in C_b(\R)$ be such that $b(x) = xg(x)$ for all $x \in \R$, where $g \in C_0(\R)$ has compactly supported Fourier transform.
Then
\[
\widetilde{\textnormal{Av}}(b(D)_{A})=\textnormal{Av}_{L^2}(b(D)_{L^2(E)}) \quad \in 
\cK(L^2_T(E)^G).
\]
\end{lemma}
	\begin{proof}
The map \eqref{eq til Av} is uniquely determined by the property that for all ${\kappa} \in A$ and all $S \in \calL_A(A)$,
\[
\Av({\kappa}S) =\Av({\kappa})\TAv(S) \quad \text{and} \quad \Av(S{\kappa}) =\TAv(S)\Av({\kappa}).
\]
In fact, $\TAv$ is already determined by these properties for ${\kappa}$ in the dense subalgebra  $C^{\infty}_{\ker}(M; L^2(E))^G_{\loc}$. So the claim is that for all ${\kappa} \in C^{\infty}_{\ker}(M; L^2(E))^G_{\loc}$,
\beq{eq Avbj}
\begin{split}
\Av({\kappa}b(D)_{A}) &=\Av({\kappa})\Av_{L^2}(b(D)_{L^2}) ;\\
\Av(b(D)_{A} {\kappa}) &=\Av_{L^2}(b(D)_{L^2})  \Av({\kappa}).
\end{split}
\eeq

The second  equality in \eqref{eq Avbj} is true, because by Lemmas \ref{lem Av op ker} and \ref{lem fc Hilb cstar},
\[
\Av_{L^2}(b(D)_{L^2})  \Av({\kappa}) = \Av(b(D)_{L^2} {\kappa}) =\Av(b(D)_{A} {\kappa}). 
\]

The element ${\kappa}b(D)_{A}$ is defined as
\[
L_{\kappa} \circ b(D)_{A} \in \cK_{A}(A) \cong A.
\]
Here $L_{\kappa}$ is left composition with ${\kappa}$. Lemma \ref{lem fc Hilb cstar} implies that for all ${\kappa},{\kappa}'  \in C^{\infty}_{\ker}(M; L^2(E))^G_{\loc}$, the element ${\kappa}b(D)_{A} {\kappa}'$ has a smooth kernel, equal to the composition of the kernels of ${\kappa}$ and of $b(D)_{L^2} {\kappa}'$. By associativity, that equals the composition of the smooth kernels of ${\kappa}b(D)_{L^2}$ and ${\kappa}'$. So ${\kappa}b(D)_{A}  \in C^{\infty}_{\ker}(M; L^2(E))^G_{\loc}$, and its kernel is the smooth kernel of  ${\kappa}b(D)_{L^2}$. Hence, by Lemma \ref{lem Av op ker},
\begin{align*}
\Av({\kappa}b(D)_{A}) = \Av({\kappa}b(D)_{L^2})=\Av({\kappa})\Av_{L^2}(b(D)_{L^2}).&\qedhere
\end{align*}
\end{proof}

%Let $b \in C_b(\R)$ be a normalising function as below Proposition \ref{prop Roe max}. That proposition implies that
%\[
%\tilAv(b(D))^2 - 1  = \Av(b(D)^2 - 1)\in \cK(L^2_T(E)^G).
%\]
%So $\tilAv(b(D))$ is a Fredholm operator.

\section{The invariant index} \label{sec pf invar index}

In this section, we use the averaging maps from Section \ref{sec avg} to prove Theorem \ref{thm invar index}. 

\subsection{The index of $\tilAv(b(D))$}

Let $b$ be a normalising function as below Proposition \ref{prop Roe max}.
That proposition implies that
\[
\tilAv(b(D))^2 - 1  = \Av(b(D)^2 - 1)\in \cK(L^2_T(E)^G).
\]
So $\tilAv(b(D))$ is a Fredholm operator. 

We will prove Theorem \ref{thm invar index} by proving  Propositions \ref{prop til Av invar}
 and \ref{prop til Av I} below. Here, to be precise, by the index of the odd-graded operator  $\tilAv(b(D))$ on $L^2_T(E)^G$, we mean the index in the graded sense; i.e. the index of its restriction $\tilAv(b(D))_+$ to even-graded sections.

\begin{proposition} \label{prop til Av invar}
\[
\ind(\tilAv(b(D))) = \ind(D)^G.
\]
\end{proposition}

\begin{proposition} 
\label{prop til Av I}
\[
\ind(\tilAv(b(D))) = I_*(\ind_G^{\loc}(D)).
\]
\end{proposition}
\begin{proof}
Consider the boundary maps
\[
\begin{split}
\partial_{\cB}\colon K_1(\cB(L^2_T(E)^G)/\cK(L^2_T(E)^G)) & \to K_0(\cK(L^2_T(E)^G)); \\
\partial_{A}\colon K_1(\calL_A(A)/\cK_A(A))& \to  K_0(\cK_A(A)).
\end{split}
\]
Naturality of boundary maps with respect to $*$-homomorphisms implies that
\[
\ind(\tilAv(b(D))) = \partial_{\cB}[\tilAv(b(D))_+] = \Av(\partial_A[b(D)_+])
\]
Here we used the fact that $\tilAv(b(D))_+ = \tilAv(b(D)_+)$. Proposition \ref{prop Avpi pi} and \eqref{eq Cker inj lim} now imply that the right hand side equals
\begin{align*}
I_*(\partial_A[b(D)_+] \oplus 0) = I_*(\ind_G^{\loc}(D)).&\qedhere
\end{align*}
\end{proof}

To prove Theorem \ref{thm invar index}, it remains to prove Proposition \ref{prop til Av invar}.

\subsection{The index of $\tilAv(b(D))$ and the invariant index}

In this subsection, we prove Proposition \ref{prop til Av invar}. The main point of the proof  is dealing with the fact that sections in $L^2_T(E)^G$ are not square integrable in general.

We may choose the normalising function $b$ so that $b(t) = \cO(t)$ as $t\to 0$. Then the operator
\[
S := \frac{b(D)}{D}
\]
on $L^2(E)$ is bounded. 
	\begin{lemma} \label{lem s sigma}
		For all $\sigma\in \Gamma_c^\infty(E)$ and $s\in \Gamma^\infty(E)\cap L^2_T(E)^G$, the inner product $(Ds,S\sigma)_{L^2(E)}$ is well-defined and equals
		$$(\widetilde{\textnormal{Av}}(b(D))s,\sigma)_{L^2(E)}.$$
	\end{lemma}
%\Todo: how does $\TAv(b(D))$ act on sections of $E$?	
	\begin{proof}
As in Subsection \ref{sec Av Hilb Cstar}, we use subscripts $A$ and $L^2$ to distinguish functional calculus of operators on the Hilbert $A$-module $A$ and on $L^2(E)$.
	
Let $(b_j)_{j=1}^{\infty}$ be a sequence in $C_b(\R)$ converging to 	$b$ in the $\sup$-norm, such that for each $j$, and all $x \in \R$, $b_j(x) = xg_j(x)$, where $g_j \in C_0(\R)$ has compactly supported Fourier transform. Then $g_j(D)_{L^2}$, and hence $b_j(D)_{L^2}$, has finite propagation. So by Lemma \ref{lem fin prop est norm},
for all  $s\in\Gamma_{\textnormal{tc}}^G$ and $\sigma\in\Gamma_c(E)$,
\beq{eq est inner prod}
|(\textnormal{Av}_{L^2}(b_j(D)_{L^2})s,\sigma)_{L^2(E)}|\leq\norm{b_j(D)_{L^2}}|(s,\sigma)_{L^2(E)}|.
\eeq
We also have
\beq{eq norms bjD}
\norm{b_j(D)_{L^2}}_{}\leq\norm{b_j(D)_A}_{\calL_A(A)}
\eeq
To see that this is true, note that %$a:= b_j(D) \in C^*_{\ker}(M; Z, L^2(E))$. And 
by Lemma \ref{lem fc Hilb cstar}, 
\[
b_j(D)_{L^2} = b_j(D)_{A} =: \kappa \in C^*_{\ker}(M; L^2(E))^G_{\loc}.
\]
And
\[
\|\kappa\|_{\cB(L^2(E))} = \|\kappa \oplus 0\|_{\cB(L^2(E) \otimes L^2(G))} \leq \|\kappa \oplus 0\|_{\max}
 = \|\kappa\|_{A} = \|\kappa\|_{\calL_A(A)}.
\]
The inequality is true, because the defining representation of $\cB(L^2(E) \otimes L^2(G))$ in $L^2(E) \otimes L^2(G)$ trivially restricts to a $*$-representation of $C^*_{\ker}(M; L^2(E))^G_{\loc}$. So \eqref{eq norms bjD} follows.
	
 By the first point of Theorem \ref{thm:functionalcalculus},  $$b_j(D)\rightarrow b(D)\in\calL_A(A).$$
		By Lemma \ref{lem:Avsequal},
		$$\widetilde{\textnormal{Av}}(b_j(D))=\textnormal{Av}_{L^2}(b_j(D))\in\mathcal{B}(L^2_T(E))^G.$$
		This equality, together with \eqref{eq est inner prod} and \eqref{eq norms bjD} implies that
		$$|(\widetilde{\textnormal{Av}}(b_j(D))s,\sigma)_{L^2(E)}|\leq\norm{b_j(D)}_{\calL_A(A)}|(s,\sigma)_{L^2(E)}|$$
		and hence
		$$(\widetilde{\textnormal{Av}}(b(D))s,\sigma)_{L^2(E)}=\lim_{j\rightarrow\infty}(\widetilde{\textnormal{Av}}(b_j(D))s,\sigma)_{L^2(E)}.$$
		This means it suffices to prove the claim for each finite-propagation approximant $b_j(D)$, namely that
		$$(\widetilde{\textnormal{Av}}(b_j(D))s,\sigma)_{L^2(E)}=(\textnormal{Av}_{L^2}(b_j(D))s,\sigma)_{L^2(E)}=(Ds,S_j\sigma)_{L^2(E)}.$$
	
To prove the latter equality, let $r_j>0$ be the propagation of $b_j(D)$, and let the partition of unity $\{\psi_k\}_{k=1}^{\infty}$ be as in \eqref{eq av T}. Suppose these functions are real-valued. For $\sigma \in \Gamma^{\infty}_c(E)$, only finitely many of the functions $\psi_k$ have supports closer than $r_j$ to $\supp(\sigma)$. Let $K_{\sigma}\subset \N$ be the set of the corresponding indices $k$. Then for $s\in \Gamma^{\infty}(E)\cap L^2_T(E)^G$, since $b_j(D)$ commutes with finite sums, 
\begin{multline*}
(\Av(b_j(D))s, \sigma)_{L^2(E)} = (b_j(D) \sum_{k \in K_{\sigma}} \psi_ks, \sigma)_{L^2(E)} \\
= (s, \sum_{k \in K_{\sigma}} \psi_k  b_j(D) \sigma)_{L^2(E)} 
= (s,  b_j(D) \sigma)_{L^2(E)}
= (Ds, S_j\sigma)_{L^2(E)}.
\end{multline*}
\end{proof}

\begin{lemma} \label{lem ell reg}
\[
\ker(\TAv(b(D)) \subset \Gamma^{\infty}(E)
\]
\end{lemma}
\begin{proof}
%We first claim that 
%\beq{eq D TAv}
%D \TAv(b(D)) =\TAv(Db(D)) =  \TAv(b(D)) D
%\eeq
%as unbounded operators on $L^2_T(E)^G$. Indeed, the first equality  means that for all $a \in C_{\ker, \infty}(M; Z, L^2(E))^G$,
%\[
%\Av(Db(D)a) = D \Av(b(D)a).
%\]
%The averaging operation on smooth kernels commutes with $D$, so this equality indeed holds. The second equality in \eqref{eq D TAv} holds for similar reasons.
%
Let $s \in \ker(\TAv(b(D))$. Then % \eqref{eq D TAv} implies that 
%for all $n \in \N$,
\[
s = \Av(1-b(D)^2)s.
\]
Here we used the fact that $1-b(D)^2 \in C_{\ker}^{\infty}(M;  L^2(E))^G_{\loc}$. That also implies that the right hand side is smooth, and hence so is $s$. 
\end{proof}

\begin{lemma} \label{lem dense}
The space $S(\Gamma_c^{\infty}(E))$ is dense in $L^2(E)$.
\end{lemma}
\begin{proof}
%This is a simple consequence of boundedness and invertibility of $S$ and density of $\Gamma_c^{\infty}(E)\subset L^2(E)$. 
Let $s\in L^2(E)$, and let $(s_j)_{j=1}^{\infty}$ be a sequence in $\Gamma_c^{\infty}(E)$ converging to $s$ in $L^2$-norm. Then
\[
\|Ss - Ss_j\|_{L^2(E)}\to 0.
\]
So $S(\Gamma_c^{\infty}(E))$ is dense in $\im(S)$.
Now if $t \in \Gamma_c^{\infty}(E)$, then $t \in\dom(S^{-1})$. Hence $t = S(S^{-1}t) \in \im(S)$. So $\im(S) \subset L^2(E)$ is dense, which completes the proof.
\end{proof}
	
\begin{proof}[Proof of Proposition \ref{prop til Av invar}.]
By elliptic regularity, $\ker(D) \subset \Gamma^{\infty}(E)$. So
Lemma \ref{lem s sigma} implies that
\[
\ker(D) \cap L^2_T(E)^G \subset \ker(\TAv(b(D))).
\]

We claim that for all $s\in \ker (\tilAv(b(D)))\cap \Gamma^{\infty}(E)$ and $\sigma \in \Gamma^{\infty}_c(E)$, 
\beq{eq Ds sigma L2}
(Ds, \sigma)_{L^2(E)}=0.
\eeq
Indeed,
Let  $\sigma \in \Gamma^{\infty}_c(E)$. By Lemma \ref{lem dense}, there is a sequence $(s_j)_{j=1}^{\infty}$ in $\Gamma_c^{\infty}(E)$ such that $(Ss_j)_{j=1}^{\infty}$ converges to $\sigma$ in $L^2$-norm. 
For all $s\in \ker \tilAv(b(D))\cap \Gamma^{\infty}(E)$, Lemma \ref{lem s sigma} implies that $(Ds, Ss_j)_{L^2(E)} = 0$. So \eqref{eq Ds sigma L2} follows.

By Lemma \ref{lem ell reg}, \eqref{eq Ds sigma L2} implies that
\[
 \ker(\TAv(b(D))) \subset \ker(D) \cap L^2_T(E)^G.
\]
So
\[
 \ker(\TAv(b(D))) =\ker(D) \cap L^2_T(E)^G,
\]
including gradings.
\end{proof}

\section{Quantisation commutes with reduction} \label{sec pf QR=0 1}

In this section, we use Theorem \ref{thm invar index} to prove Theorem \ref{thm QR=0}.

\subsection{A localisation estimate} \label{sec loc est}

 Let $U$ be a relatively cocompact, $G$-invariant neighbourhood of $\mu^{-1}(0)$. Since $U$ is relatively cocompact, we can enlarge $Z$, outside which the estimate \eqref{eq Callias est} holds,  if needed, so that its interior contains the closure of $U$. 

Fix a $G$-invariant metric on $M \times \kg \to M$, where $G$ acts on $\kg$ via the adjoint action. Let $\|\mu\|^2$ be the square of the  norm of the $\Spinc$-moment map $\mu$ with respect to this metric. Let $v^{\mu}$ be the vector field on $M$ induced by the map $M \to \kg$ dual to $\mu$ with respect to this metric. Explicitly, if $m \in M$ and $X \in \kg$ is dual to $\mu(m) \in \kg^*$ for the metric at $m$, then 
\[
v^{\mu}(m) = X^M(m).
\]
By $G$-invariance of the metric on $M \times \kg$, this vector field is $G$-invariant. Let $f \in C^{\infty}(M)^G$ be a nonnegative function with cocompact support,\footnote{In earlier work on deformed Dirac operators of the form $D + ifc(v^{\mu})$ on non-cocompact manifolds \cite{Braverman02, Braverman14, Mathai13, HM14, HS1}, the function $f$ was required to grow at infinity in a suitable way. In our setting, we actually need $f$ to vanish outside a cocompact set (this is used in the proof of Proposition \ref{prop loc est}). This is possible, because invertibility at infinity is guaranteed by the term $\Phi$.} such that $f \equiv 1$ on $Z$.
For any $G$-invariant real-valued $h \in C^{\infty}(M)$, and any $p \in \Z_{\geq 0}$ and $t>0$, consider the operator
\[
D_{p,h, t} := \tilde D^{L^p} + t(h\Phi -ifc(v^{\mu}))
\]
on $\Gamma^{\infty}(\cS \otimes L^p)$.

Let $\chi\in C^{\infty}(M)$ be a cutoff function for the action by $G$ on $M$, as in \eqref{eq chi cutoff}. For $s \in \Gamma^{\infty}_{\tc}(\cS \otimes L^p)^G$, define
\[
 \|\chi s\|_{W^1(\cS \otimes L^p)}^2 :=  \|\chi \tilde D^{L^p}  s\|_{L^2(\cS \otimes L^p)}^2  +  \|\chi s\|_{L^2(\cS \otimes L^p)}^2 
\]

The proof of Theorem \ref{thm QR=0} is based on the following localisation property. This is an analogue of Theorem 2.1 in  \cite{Zhang98}.
% Let $c$ be as in \eqref{eq Callias est}. 
Let $U'$ be a $G$-invariant neighbourhood of $\mu^{-1}(0)$ whose closure is contained in $U$.
\begin{proposition}\label{prop loc est}
There are  constants $C,b>0$, a function $h\in C^{\infty}(M)^G$, equal to a constant greater than or equal to $1$ outside $U$ and constant $0$ inside $U'$,
and a $p_0 \in \Z_{\geq 0}$, such that for all $p \geq p_0$, $t\geq 1$ and
 all $s \in \Gamma^{\infty}_{\tc}(\cS \otimes L^p)^G$ supported outside $U$,
\beq{eq loc est}
\|\chi D_{p,h} s\|_{L^2(\cS \otimes L^p)}^2 \geq C\bigl( \|\chi s\|_{W^1(\cS \otimes L^p)}^2 + (t-b) \|\chi s\|_{L^2(\cS \otimes L^p)}^2 \bigr).
\eeq
\end{proposition}

Simlarly to \cite{Mathai13, HM14, Mathai10, Zhang98}, the proof of Proposition \ref{prop loc est} is based on an expression for squares of deformed operators, Proposition \ref{prop Bochner}.  This expression is deduced from an expression in \cite{HM14}.

For any  $p \in \Z_{\geq 0}$ and any $G$-equivariant differential operator $D$ on $\Gamma^{\infty}(\cS \otimes L^p)$,
  let $\hat D$ be the operator on $\chi \Gamma^{\infty}_{\tc}(\cS \otimes L^p)^G$ defined by
\[
\hat D(\chi s) = \chi D s,
\]
for all $s \in \Gamma^{\infty}_{\tc}(\cS \otimes L^p)^G$. 
Note that 
\[
\chi \Gamma^{\infty}_{\tc}(\cS \otimes L^p)^G \subset \Gamma_c^{\infty}(\cS \otimes L^p).
\]
Let $\hat D^*$ be the formal adjoint of $\hat D$  with respect to the $L^2$-inner product.
%For brevity, we write $D_p := \tilde D^{L^p}$.
\begin{proposition} \label{prop Bochner}
There is a $G$-equivariant vector bundle endomorphism $B'$ of $\cS \otimes L^p$, and
for all $t >0$, 
there is a $G$-equivariant vector bundle endomorphism $B_t$ of $\cS \otimes L^p$,  which vanishes at points where $f$ and $df$ vanish, and which satisfies the fibrewise estimate $B_t \geq tB'$ for all $t \geq 1$, 
such that, on sections in $\chi \Gamma^{\infty}_{\tc}(\cS \otimes L^p)^G$ supported outside $U$,
\beq{eq D star D}
\hat D_{p,h, t}^* \hat D_{p,h, t} = \widehat {\tilde D^{L^p}}^* \widehat {\tilde D^{L^p}} + 
ht\{ \tilde D, \Phi\} + h^2t^2\Phi^2 +
B_t + (2p+1)2 \pi t f \|\mu\|^2.
\eeq
(Here, as in \eqref{eq Spinc Callias est}, we omit `$\otimes 1_{L^p}$'.)
\end{proposition}
\begin{proof}
Corollary 8.5 in \cite{HM14} states that, on $\chi \Gamma^{\infty}_{\tc}(\cS \otimes L^p)^G$,  
\[
%\widehat{(\tilde D^{L^p} -itfc(v^{\mu}))}^*\widehat{(\tilde D^{L^p} -itfc(v^{\mu}))} = 
\hat D_{p,0,t}^*\hat D_{p,0,t} = 
\widehat{\tilde D^{L^p}}^*\widehat{\tilde D^{L^p}} + tB'' + (2p+1)2 \pi t f \|\mu\|^2 + t^2 f^2 \|v^{\mu}\|^2,
\]
for a
$G$-equivariant vector bundle endomorphism $B''$ of $\cS \otimes L^p$, which vanishes at points where $f$ and $df$ vanish. Because $h$ is constant outside $U$, this implies that, on sections in $\chi \Gamma^{\infty}_{\tc}(\cS \otimes L^p)^G$ supported outside $U$,
 the left hand side of \eqref{eq D star D} equals
\beq{eq D star D 1}
% (\widehat{ {\tilde D^{L^p}} + th \Phi})^* (\widehat{ {\tilde D^{L^p}} + th \Phi}))\\
%tB'' + (2p+1)2 \pi t f \|\mu\|^2 + t^2 f^2 \|v^{\mu}\|^2 -i ht^2f\{\Phi, c(v^{\mu})\},
%
\widehat{\tilde D^{L^p}}^*\widehat{\tilde D^{L^p}} + tB'' + (2p+1)2 \pi t f \|\mu\|^2 + t^2 f^2 \|v^{\mu}\|^2 + t^2 h^2 \Phi^2  + ht \{\tilde D, \Phi\}-i ht^2f\{\Phi, c(v^{\mu})\}.
\eeq
%for a
%$G$-equivariant vector bundle endomorphism $B''$ of $\cS \otimes L^p$, which vanishes at points where $f$ and $df$ vanish. 
%
The fact that $\{{\tilde D}, \Phi\}$ is a vector bundle endomorphism implies that $\{\Phi, c(v^{\mu})\} = 0$. % (\Todo: check or assume).

% and that 
%\[
%(\widehat{ {\tilde D^{L^p}} + th \Phi})^* (\widehat{ {\tilde D^{L^p}} + th \Phi})) = \widehat {\tilde D^{L^p}}^* \widehat {\tilde D^{L^p}} +  
%ht\{ \tilde D, \Phi\} + h^2t^2\Phi^2. 
%\]
Set
\[
\begin{split}
B' &:= B'' + f^2 \|v^{\mu}\|^2;\\
B_t &:= tB'' + t^2 f^2 \|v^{\mu}\|^2;\\
\end{split}
\]
Then \eqref{eq D star D} follows.
\end{proof}

For $h \in C^{\infty}(M)^G$, $p \in \N$ and $t>0$, write
\[
A_{p,h,t}:= \hat D_{p,h, t}^* \hat D_{p,h, t} - \widehat {\tilde D^{L^p}}^* \widehat {\tilde D^{L^p}}. 
\]
This is a vector bundle endomorphism by Proposition \ref{prop Bochner}.
\begin{lemma} \label{lem Apht}
There are a constant $C \in (0,1]$, a function $h\in C^{\infty}(M)^G$, equal to a constant greater than or equal to $1$ ouside $U$ and constant $0$ inside $U'$,
and a $p_0 \in \Z_{\geq 0}$, such that for all $p \geq p_0$, $t\geq 1$, the vector bundle endomorphism $A_{p,h,t}$ satisfies the pointwise estimate
\beq{eq est A}
A_{p,h,t} \geq Ct
\eeq
on $M \setminus U$.
\end{lemma}
\begin{proof}
Let $B'$ be as in Proposition \ref{prop Bochner}. Because it is $G$-equivariant, it is bounded on cocompact sets. Let $h  \in C^{\infty}(M)^G$ be nonnegative, such that $h|_{U'} = 0$, $h\geq 1$ outside $U$ and, on the cocompact set  $\supp(f) \setminus U$,
\beq{eq est h 2}
hc^2/2 \geq  \|B'\|.
\eeq

On the cocompact set $Z \setminus U$, the positive function $\|\mu\|^2$ is bounded below by a positive constant. And the $G$-equivariant vector bundle endomorphism $\{\tilde D, \Phi\}$ is bounded on that set as well.
Choose $p_0 \in \Z_{\geq 0}$ such that, on $Z\setminus U$,
\beq{eq choice p0}
(2p_0 + 1)2\pi \|\mu\|^2 \geq \|B'\| + h\|\{\tilde D, \Phi\}\| + 1.
\eeq

Set $C:= \min(c^2/2, 1)$, where $c$ is as in \eqref{eq Callias est}. 
We claim that the  estimate \eqref{eq est A} holds for these choices of $p_0$, $h$ and $C$ and all $t \geq 1$.

%Write
%\[
%A:= \hat D_{p,h, t}^* \hat D_{p,h, t} - \widehat {\tilde D^{L^p}}^* \widehat {\tilde D^{L^p}}. 
%\]
Let $t \geq 1$, and  let $B_t$ be as in Proposition \ref{prop Bochner}. Then $B_t = 0$ on $M \setminus \supp(f)$, so Proposition \ref{prop Bochner} implies that on that set,
\[
A_{p,h,t} = ht\{ \tilde D, \Phi\} + h^2 t^2 \Phi^2.
\]
Since $ht\geq 1$ outside $U$, the right hand side is at least equal to $tc^2$.

On $\supp(f) \setminus Z$, we similarly have
\begin{multline*}
A_{p,h,t} \geq
t^2h^2\Phi^2 +
th\{ \tilde D, \Phi\}  
+ tB' + (2p+1)2 \pi tf \|\mu\|^2 \geq \\
t^2h \Phi^2 
+ t\{ \tilde D, \Phi\} + h^{-1}tB' \geq t(c^2 - h^{-1}\|B'\|) \geq tc^2/2,
\end{multline*}
by \eqref{eq est h 2}

Finally, on $Z \setminus U$, 
\[
A_{p,h,t} \geq t\bigl(h\{ \tilde D, \Phi\} + 
B' + (2p+1)2 \pi   \|\mu\|^2 \bigr) \geq t,
\]
if $p \geq p_0$ as in \eqref{eq choice p0}. 
%
%%Because of \eqref{eq anticomm DLp Phi}, 
%%\[
%%({\tilde D^{L^p}} + h\Phi)^2 = 
%%(\tilde D^{L^p})^2 
%%+ h^2\Phi^2 +
%%h\{ \tilde D, \Phi\}.% + c(dh)\Phi .
%%\]
%%Hence
% Proposition \ref{prop Bochner} implies that, on $G$-invariant sections,
%\[
%\hat D_{p,h}^* \hat D_{p,h} = (\tilde D^{L^p})^2
%+ h^2\Phi^2 +
%h\{ \tilde D, \Phi\} %+ c(dh)\Phi 
%+ B + (2p+1)2 \pi f \|\mu\|^2.
%\]
%
%
%On $M \setminus \supp(f)$, we have $B = 0$, so 
%\[
%\hat D_{p,h}^* \hat D_{p,h} = (\tilde D^{L^p})^2 
%+ h^2\Phi^2 +
%h\{ \tilde D, \Phi\}  %+ c(dh)\Phi.
%\]
%And, since $h \geq 1$ outside $U$, we have 
%\[
% h^2\Phi^2 +
%h\{ \tilde D, \Phi\} \geq c^2
%\]
%on $M \setminus \supp(f)$.
%
%
%On $\supp(f) \setminus Z$, we similarly have
%\[
%h^2\Phi^2 +
%h\{ \tilde D, \Phi\}  
%+ B + (2p+1)2 \pi f \|\mu\|^2 \geq h \Phi^2 + \{ \tilde D, \Phi\} + h^{-1}B \geq c^2 - h^{-1}\|B\| \geq c^2/2.
%\]
%
%
%Finally, on $Z \setminus U$, the last two terms on the right hand side of \eqref{eq D star D} satisfy
%\[
%B + (2p+1)2\pi f\|\mu\|^2 = B + (2p+1)2\pi \|\mu\|^2 \geq 1
%\]
%if $p \geq p_0$ as in \eqref{eq choice p0}. 
\end{proof}

\begin{proof}[Proof of Proposition \ref{prop loc est}]
Let $h$, $p_0$ and $C$ be as in Lemma \ref{lem Apht}. Let $p\geq p_0$ and $t \geq 1$.
Let $s \in \Gamma^{\infty}_{\tc}(\cS \otimes L^p)^G$ be supported outside $U$.
 Then Lemma \ref{lem Apht} implies that 
 \[
 \begin{split}
 \|\chi D_{p,h} s\|_{L^2(\cS \otimes L^p)}^2 &\geq 
 \|\chi \tilde D^{L^p} s\|_{L^2(\cS \otimes L^p)}^2 +
 Ct \| \chi s\|_{L^2(\cS \otimes L^p)}^2 \\
 &\geq
  C\bigl( \|\chi s\|_{W^1(\cS \otimes L^p)}^2 + (t-C^{-1}) \|\chi s\|_{L^2(\cS \otimes L^p)}^2\bigr).
 \end{split}
 \]
 Here we used the fact that $C \leq 1$.
\end{proof}

\subsection{Proof of Theorem \ref{thm QR=0}} \label{sec pf QR=0}

\begin{lemma} \label{lem D+A}
Let $D$ be any operator as in Subsection \ref{sec loc index}, where $E$ is $\Z/2$-graded. Let $S \in \End(E)^G$ be an odd, fibrewise self-adjoint vector bundle endomorphism which is zero outside a cocompact set. Then $(D+S)^2$ has a uniform lower bound outside a cocompact set, and 
\[
\ind_G^{\loc}(D+S) = \ind_G^{\loc}(D) \quad \in K_0(C^*_{\max}(G)).
\]
\end{lemma}
\begin{proof}
The operator $D+S$ is elliptic, and $(D+S)^2$ has a positive lower bound outside $Z \cup \supp(S)$. Hence its index is well-defined. Since $\supp(S)$ is cocompact, $S$ is a bounded operator on $L^2(E)$. Hence the path of operators 
\[
t\mapsto b(D+tS)
\]
is continuous in the operator norm, where $b$ is a normalising function as in Subsection \ref{sec loc index}. This defines an operator homotopy showing that 
\[
[b(D+tS)_+] \in K_1(\calL_A(A)/\cK_A(A))
\]
is independent of $t$. \end{proof}

\begin{lemma} \label{lem D+hPhi}
For all constants $h \geq 1$,
\[
\ind_G^{\loc}(\tilde D + \Phi) = \ind_G^{\loc}(\tilde D + h\Phi). 
\]
\end{lemma}
\begin{proof}
Set
\[
\tilde D_t := \tilde D + (1-t+th)\Phi
\]
For all $t$, we have $(1-t+th) \geq 1$, so that
\[
\tilde D_t^2 = \tilde D^2 + (1-t+th) (\tilde D \Phi + \Phi \tilde D) + (1-t+th)^2 \Phi^2 \geq (D+\Phi)^2.
\]
This has a positive lower bound outside a fixed cocompact set. 
So, for a suitable normalising function $b$, we have an invertible element
\[
b(D_t)_+ \in \calL_A(A)/\cK_A(A)
\]
for all $t \in [0,1]$. Because  $\| \Phi \|$ is bounded, Lemma \ref{lem vb end bdd} implies that this path of operators is continuous in the operator norm. Hence the class
\[
[b(D_t)_+] \in K_1(\calL_A(A)/\cK_A(A))
\]
is independent of $t$. 
\end{proof}

%\begin{remark}
%The proofs of Lemmas \ref{lem D+S} and \ref{lem D+hPhi} in fact show that the classes defined by the operators in question in localised $K$-homology are equal. 
%\end{remark}

The following consequence of Propsition \ref{prop loc est} is an important step in the proof of Theorem \ref{thm QR=0}. It is an analogue of Lemma 6.12 in \cite{Mathai13}.
\begin{lemma} \label{lem Elambda}
Let $p_0$ and $h$ be as in Proposition \ref{prop loc est}.
For all $\lambda > 0$, there exists $t_0 > 0$ such that for all $t \geq t_0$, the intersection of the interval $(-\infty, \lambda]$ with the spectrum of $D_{p,h,t}^2$ as an unbounded, self-adjoint operator on $L^2_T(E)^G$ is discrete, with finite-dimensional eigenspaces.
\end{lemma}
\begin{proof}
This follows from Proposition \ref{prop loc est} in exactly the same way that Lemma 6.12 in \cite{Mathai13} follows from Proposition 6.3 in that paper.
\end{proof}

\begin{proof}[Proof of Theorem \ref{thm QR=0}]
Let $\mu$ be as in Subsection \ref{sec loc est}.
Let $h$ and $p_0$ be as in Proposition \ref{prop loc est}, and fix $p\geq p_0$. 
Let $t_0$ be as in Lemma \ref{lem Elambda}, and fix $t\geq t_0$.
Lemmas 
 \ref{lem D+A} and \ref{lem D+hPhi} imply that
 \[
\ind_G^{\loc}(D_p) = \ind_G^{\loc}(D_{p,h,t}).
\]
%And
%by Lemma \ref{lem D+A} and the fact that $f$ has cocompact support, 
%\[
%\ind_G^{\loc}(\tilde D^{L^p} + h\Phi) = \ind_G^{\loc}(D_{p,h}).
%\]
So by Theorem \ref{thm invar index}, 
\[
I_*(\ind_G^{\loc}(D_p) ) 
%I_*(\ind_G^{\loc}(\tilde D^{L^p} + h\Phi) )
= \ind(D_{p,h,t})^G.
\]

From this point onwards,
one proves that 
\[
\ind(D_{p,h,t})^G = \ind(D_{M_0}^{L_0^p})
\]
exactly following
% the proof of Theorem \ref{thm QR=0} is exactly as the proof of Theorem 3.6 in \cite{Mathai13},  in 
 Section 7 of \cite{Mathai13}, where 
Proposition \ref{prop loc est} and Lemma \ref{lem Elambda} in this paper should be substituted for 
Proposition 6.3 and  Lemma 6.12  in \cite{Mathai13}, respectively. Furthermore, we use the fact that on the set $U'$,
\[
D_{p,h,t} =  \tilde D^{L^p}   -itc(v^{\mu}),
\]
which is exactly the operator $D^{L^p}_{t}$ in \cite{Mathai13}, by Lemma 2.3 in \cite{Mathai13}. (The minus sign is caused by a different sign convention in the definition of vector fields induced by Lie algebra elements; compare \eqref{eq def XM} to (20) in \cite{Mathai13}.)
%
%Proposition \ref{prop loc est} implies that the index on the right hand side localises near $\mu^{-1}(0)$ in the sense of Proposition 3.1 in \cite{Mathai10}, Proposition 6.3 in \cite{Mathai13} and Proposition 8.2 in \cite{HM14}. Hence the result follows as in Section 7 of \cite{Mathai13}. The only difference is the addition of the vector bundle endomorphism $\Phi$, but this does not affect the index on $M_0$. Indeed, Lemma \ref{lem D+A} implies that we may replace $\Phi$ by an endomorphism that is zero near $\mu^{-1}(0)$ and equal to $\Phi$ outside a cocompact set, without changing the index.
\end{proof}

\bibliographystyle{plain}

\bibliography{mybib}

\end{document}